\documentclass[10pt, reqno]{amsart} 
\usepackage{amsmath,amssymb,amsfonts,amscd}
\usepackage[margin = 1.25in]{geometry}
\usepackage{graphicx}  
 \usepackage[all]{xy}
\usepackage[usenames,dvipsnames]{xcolor}
\usepackage{mathrsfs}
\usepackage{tikz-cd} 
\usepackage{stmaryrd}
\usepackage{hyperref}
\usepackage{verbatim}
\usepackage{blkarray}
\usepackage{amssymb} 
\def\acts{\curvearrowleft}
 
\usepackage[toc,page]{appendix}
 
\newcommand{\F}{\mathbf{F}}

\newcommand{\CC}{\mathbf{C}}
\newcommand{\G}{\mathbf{G}}

\newcommand{\wt}[1]{\widetilde{#1}}

\newcommand{\Q}{\mathbf{Q}}
\newcommand{\Z}{\mathbf{Z}}

\newcommand{\mf}[1]{\mathfrak{#1}}

\newcommand{\Gal}{\operatorname{Gal}}

\newcommand{\R}{\mathbf{R}}

\newcommand{\ul}[1]{\underline{#1}}
\newcommand{\ol}[1]{\overline{#1}}
\newcommand{\wh}[1]{\widehat{#1}}

\newcommand{\mbb}[1]{\mathbb{#1}}

\newcommand{\Cal}[1]{\mathcal{#1}}
\newcommand{\A}{\mathbf{A}}

\newcommand{\co}{\colon}
\newcommand{\mrm}[1]{\mathrm{#1}}
\newcommand{\msf}[1]{\mathsf{#1}}

\newcommand{\bs}{\backslash}

\newcommand{\TT}{\mathbf{T}}
\newcommand{\LL}{\mathbb{L}}
\newcommand{\scr}[1]{\mathscr{#1}}

\newcommand{\surj}{\twoheadrightarrow}
\newcommand{\dotimes}{\stackrel{\mrm{L}}{\otimes}}
\newcommand{\bu}{\bullet}

\DeclareMathOperator{\GL}{GL}

\DeclareMathOperator{\Frob}{Frob}
\DeclareMathOperator{\ab}{ab}

\DeclareMathOperator{\Hom}{Hom}

\DeclareMathOperator{\rank}{rank}

\DeclareMathOperator{\Rep}{Rep}

\DeclareMathOperator{\Spec}{Spec\,}

\DeclareMathOperator{\Lie}{Lie}
\DeclareMathOperator{\End}{End}

\DeclareMathOperator{\Frac}{Frac}

\DeclareMathOperator{\Stab}{Stab}
\DeclareMathOperator{\Bun}{Bun}
\DeclareMathOperator{\Ext}{Ext}

\DeclareMathOperator{\Ad}{Ad}
\DeclareMathOperator{\Gr}{Gr}

\DeclareMathOperator{\Sym}{Sym}

\DeclareMathOperator{\pt}{pt}
\DeclareMathOperator{\Fib}{Fib}

\DeclareMathOperator{\Hecke}{Hk}
\DeclareMathOperator{\Sat}{Sat}

\DeclareMathOperator{\res}{res}

\DeclareMathOperator{\Hk}{Hk}
\DeclareMathOperator{\loc}{loc}
\DeclareMathOperator{\glob}{glob}

\DeclareMathOperator{\QCoh}{QCoh}
\DeclareMathOperator{\LocSys}{LocSys}
\DeclareMathOperator{\Ran}{Ran}
\DeclareMathOperator{\Dmod}{Dmod}

\DeclareMathOperator{\IndCoh}{IndCoh}
\DeclareMathOperator{\Nilp}{Nilp}

\DeclareMathOperator{\spec}{spec}
\DeclareMathOperator{\crys}{crys}
\DeclareMathOperator{\Tor}{Tor}
\DeclareMathOperator{\Lift}{Lift}
\DeclareMathOperator{\ur}{ur}
\DeclareMathOperator{\aut}{aut}

\usepackage{color}
\newcommand{\tony}[1]{{\color{blue} \sf
    $\spadesuit\spadesuit\spadesuit$ TONY: [#1]}}

\DeclareFontEncoding{OT2}{}{} %

\newtheorem{thm}{Theorem}[section]
\newtheorem{lemma}[thm]{Lemma}
\newtheorem{prop}[thm]{Proposition}
\newtheorem{cor}[thm]{Corollary}

\theoremstyle{remark}
\newtheorem{remark}[thm]{Remark} 
\newtheorem{defn}[thm]{Definition}
\newtheorem{example}[thm]{Example}

\makeatletter
\def\th@remark{%
  \thm@headfont{\bfseries}%
  \normalfont 
  \thm@preskip \thm@preskip 
  \thm@postskip\thm@preskip
}
\def\imod#1{\allowbreak\mkern5mu({\operator@font mod}\,\,#1)}
\makeatother

\numberwithin{equation}{subsection}

\title{The spectral Hecke algebra}
\author{Tony Feng}

\begin{document} 

\begin{abstract}
We introduce a derived enhancement of local Galois deformation rings that we call the 	``spectral Hecke algebra", in analogy to a construction in the Geometric Langlands program. This is a Hecke algebra that acts on the \emph{spectral} side of the Langlands correspondence, i.e. on moduli spaces of Galois representations. We verify the simplest form of \emph{derived} local-global compatibility between the action of the spectral Hecke algebra on the derived Galois deformation ring of Galatius-Venkatesh, and the action of Venkatesh's derived Hecke algebra on the cohomology of arithmetic groups. 
\end{abstract}

\maketitle

\tableofcontents

\section{Introduction}

\subsection{Motivation}
Venkatesh and collaborators have recently introduced a number of objects -- the local derived Hecke algebra, the global derived Hecke algebra, and the (global) derived Galois deformation ring -- in order to study algebraic structures in the cohomology of locally symmetric spaces \cite{V}, \cite{PV}, \cite{GV}. However, it was suspected that there was a missing chapter in this story, which should fill in the entry ``???" in the table below. 

\begin{center}
\begin{tabular}{| l | l  l |}
\hline
& \textbf{Automorphic} & \textbf{Galois} \\
\hline 
\textbf{Local}  & \scalebox{0.9}{derived Hecke algebra} & ??? \\
\textbf{Global} &\scalebox{0.9}{cohomology of locally symmetric space} & \scalebox{0.9}{derived Galois deformation ring} \\
\hline
\end{tabular}
\end{center}

The purpose of this paper is to suggest an answer, which we call the ``spectral Hecke algebra'', that fills in this lacuna. As the table suggests, the spectral Hecke algebra is an object that ``acts" on the derived Galois deformation functor of \cite{GV}, in a manner parallel to the action of the (local) derived Hecke algebra on the cohomology of locally symmetric spaces.

\subsection{The idea of the construction}
The spectral Hecke algebra takes it name and construction from Geometric Langlands theory, which predicts a relation between the moduli stack of $G$-bundles on a complex curve $X$, and the moduli stack of $\wh{G}$-local systems on $X$. These are the analogues of the ``automorphic side'' and ``Galois side'', respectively, of the (arithmetic) Langlands correspondence, which predicts a relation between automorphic representations of $G$ and Galois representations into $\wh{G}$. A key aspect of this correspondence is the \emph{local-global compatibility}, which in a minimalistic form asks for ``Hecke eigenvalues'' of an automorphic representation to match the ``Frobenius eigenvalues'' of the corresponding Galois representation. 
	
In Geometric Langlands one still has a notion of Hecke operators, but of course there is no ``Frobenius'', so how does one formulate local-global compatibility in that context? The answer is that there is also a notion of ``Hecke operator'' on the moduli stack of local systems, coming from an object called the ``spectral Hecke stack'' \cite[\S 12.3]{AG15}. Its definition can be phrased to appear completely symmetric to that of the Hecke stack on the automorphic side. 
\begin{itemize}
\item The automorphic Hecke stack, informally speaking, classifies 

\begin{quote}
``Two $G$-bundles on a disk (around a point of the curve), together with an isomorphism of their restrictions to the punctured disk''.
\end{quote}

\item The spectral Hecke stack, informally speaking, classifies

\begin{quote}
``Two $\wh{G}$-local systems on a disk (around a point of the curve), together with an isomorphism of their restrictions to the punctured disk''.
\end{quote}

\end{itemize}

Although these descriptions seem parallel, they are qualitatively quite different: the second description is highly \emph{redundant}, because the isomorphism of the restrictions to the punctured disk must \emph{automatically} extend to the entire disk. Therefore, if one interprets the definition na\"{i}vely, it is just the same information as that of a single $\wh{G}$-local system (and no additional structure). However, if one interprets the definition in a \emph{derived} way, then the resulting derived enhancement admits an interesting action on the moduli space of global $\wh{G}$-local systems. The ``local-global compatibility'' in the context of Geometric Langlands stipulates that this action should be compatible with the action of the automorphic Hecke stack on the moduli stack of global $G$-bundles.

In the arithmetic context, the object analogous to the spectral Hecke stack should classify 
\begin{quote}
``two $\pi_1(\Z_{q})$-representations, together with an isomorphism of their restrictions to $\pi_1(\Q_{q})$.''
\end{quote}
Again it is clear that this is redundant when interpreted na\"{i}vely, but again we can interpret it in a derived way, as follows. The space of $\pi_1(\Z_{q})$-representations can be viewed as a closed substack of the space of $\pi_1(\Q_{q})$-representations, and we can form its \emph{derived} self-intersection, which will be a derived stack. The spectral Hecke algebra is obtained by performing this type of construction on framed (so as to obtain something representable) Galois deformation rings.

\subsection{What is done in this paper?} 

The main objectives of this paper are to:
\begin{enumerate}
\item Define the spectral Hecke algebra, and construct a \emph{co-action} of it on the derived Galois deformation ring from \cite{GV}. 
\item Compare the co-action of the spectral Hecke algebra on the derived Galois deformation ring  with the action of the derived Hecke algebra on the cohomology of arithmetic groups, which was studied in \cite{V}. Informally speaking, our results show that these two actions are ``compatible'' in a manner analogous to the formulation of local-global compatibility in Geometric Langlands. 
\end{enumerate}

We now introduce some notation in order to state our findings more precisely. 

\subsubsection{The automorphic side} Let $G$ be a split, semisimple, simply connected group over $\Q$. We have a system of locally symmetric spaces $Y(K)$ for $G$, indexed by the level structure $K \subset G(\A_{\Q})$. Let $\TT_K$ be the Hecke algebra acting on $H^*(Y(K); \Z_p)$, generated by Hecke operators at ``good primes''. 

We view $H^*(Y(K); \Z_p)$, and more precisely the Hecke eigensystems it carries, as an incarnation of ``automorphic forms''. Let $\chi \co \TT_K \rightarrow \ol{\Q}$ be a tempered character of $\TT_K$, and $\mf{m} = \ker \chi$. The completion $H^*(Y(K); \ol{\Q})_{\mf{m}}$ is known to be supported in a band of degrees $[j_0, j_0 + \delta]$, where $\delta = \rank G(\R)-\rank K_{\infty}$ and $j_0$ is such that $2j_0 + \delta = \dim Y(K)$. (The integers $j_0$ and $\delta$ are typically called $q_0$ and $\ell_0$ in the literature, following \cite{CG18}.) Moreover, it enjoys the following suggestive numerology:
\[
\dim_{\ol{\Q}} H^{j_0+j}(Y(K); \ol{\Q})_{\mf{m}}  = k \binom{\delta}{j} \text{ for some $k>0$}.
\]

After passing to a finite extension $\Cal{O}/\Z_p$ containing the values of $\chi$, we can consider the completion $H^*(Y(K); \Cal{O})_{\mf{m}}$. Following \cite{GV} and \cite{V}, we restrict our attention to primes $p$ where the cohomology $H^*(Y(K); \Cal{O})_{\mf{m}}$ is particularly nice. In particular, we assume that there are ``no congruences at $p$'' (which, in particular, implies $k=1$), and that this cohomology is torsion-free; see \S \ref{ssec: automorphic side} for the details. Under these assumptions, we even have that $H^*(Y(K); \Cal{O})_{\mf{m}}$ is free over $\Cal{O}$, and that 
\begin{equation}\label{eq: multiplicity numerology}
\rank_{\Cal{O}} H^{j_0+j}(Y(K); \Cal{O})_{\mf{m}}  =  k \binom{\delta}{j}.
\end{equation}

Under these assumptions, Venkatesh shows in \cite{V} that this spread of the eigensystem $\mf{m}$ in cohomological degrees can be accounted for by a \emph{derived} Hecke action. More precisely, he studies (local) derived Hecke algebras $\Cal{H}_q$ indexed by certain (Taylor-Wiles) primes $q$, and shows that their action on the lowest degree cohomology $H^{j_0}(Y(K);\Cal{O})_{\mf{m}} \otimes_{\Cal{O}} \Lambda$ generates the entirety of $H^{*}(Y(K); \Cal{O})_{\mf{m}}\otimes_{\Cal{O}} \Lambda $ for any finite quotient $\Lambda$ of $\Cal{O}$. In this way he produces (see \ref{sssec: global DHA} for more detail) a ``global derived Hecke algebra'' $\wt{\mbb{T}}_{\mf{m}} \subset \End_{\Cal{O}}(H^*(Y(K); \Cal{O})_{\mf{m}} )$ whose action on $H^{*}(Y(K); \Cal{O})_{\mf{m}} $ is free. 


\subsubsection{The Galois side} Conjecturally, the Hecke eigensystem $\mf{m}$ should correspond to a Galois representation $\rho\co \Gal(\ol{\Q}/\Q) \rightarrow \wh{G}(\Cal{O})$. This is now known in many cases; for us the most important example (since it has $\delta>0$) is that of the Weil restriction of $\GL_n$ from a CM field\footnote{Admittedly, this doesn't satisfy our semisimplicity and splitness assumptions.}, which is established in \cite{HLTT} and \cite{Sch15}. We assume the existence of $\rho$, following \cite{GV}. 

We impose niceness assumptions on the residual representation $\ol{\rho}$, in particular that it has ``big image'' and is Fontaine-Laffaille at $p$, and enjoys a strong form of local-global compatibility; see \S \ref{ssec: galois side} for the details. Again, these conditions should conjecturally be true for all sufficiently large $p$. In the case of the Weil restriction of $\GL_n$ from a CM field, they are almost all known by \cite{ten}. 

An idea going back to Mazur is to study the formal deformation functor of $\ol{\rho}$, which is representable by a ``Galois deformation ring'' $\mrm{R}_S$ \cite{Maz89}. The Taylor-Wiles method, which is at the heart of all work on modularity, centers around the relationship between the Hecke algebra $(\TT_K)_{\mf{m}}$ and $\mrm{R}_S$. However, for general groups (e.g. whenever $\delta>0$) these rings are not ``big enough'' to run the Taylor-Wiles method. Calegari-Geraghty proposed a derived enhancement of the Taylor-Wiles method in order to overcome this difficulty \cite{CG18}.

In \cite{GV}, Galatius-Venkatesh re-interpreted the Calegari-Geraghty method in terms of a \emph{derived} Galois deformation ring $\Cal{R}_S$. This is a simplicial commutative ring, whose set of connected components recovers $\mrm{R}_S$. In general, given a simplicial commutative ring $\Cal{R}$ one can form its homotopy groups $\pi_*(\Cal{R})$, which have the structure of a graded algebra. Galatius-Venkatesh show, under the assumptions mentioned above, that $\pi_* (\Cal{R}_S)$ is an exterior algebra on a free $\Cal{O}$-module of rank $\delta$, and construct an action of $\pi_*(\Cal{R}_S)$ on $H_*(Y(K); \Cal{O})_{\mf{m}}$, which realizes the latter as a free module of rank one over $\pi_*(\Cal{R}_S)$, on any generator in degree $j_0$. This gives an ``explanation'' for the numerology \eqref{eq: multiplicity numerology}. 


Note that $\pi_* (\Cal{R}_S)$ is homologically graded. Letting $ \pi_* (\Cal{R}_S)^*$ denote its $\Cal{O}$-dual, under the running assumptions, \cite[Proposition 8.6]{V} and \cite[eqn. (15.4)]{GV} give a canonical isomorphism of $\Cal{O}$-modules, 
 \begin{equation}\label{eq: global langlands isom}
 \pi_* (\Cal{R}_S) \xrightarrow{\sim} \wt{\mbb{T}}_{\mf{m}}
 \end{equation}
 
\subsubsection{Summary of results} 
We say that a ``good'' prime $q$ is a \emph{Taylor-Wiles prime for $\rho$} if $q \equiv 1 \pmod{p}$, and the image of $\Frob_q$ under the residual representation $\ol{\rho}$ is strongly regular\footnote{This omits the Selmer condition that is sometimes also included in the condition of being a "Taylor-Wiles prime".}.  In this paper we define for each Taylor-Wiles prime $q$ a \emph{spectral Hecke algebra} $\Cal{S}_q^{\Hk}$, which is a simplicial commutative ring that serves as a spectral counterpart to the derived Hecke algebras $\Cal{H}_q$. (We could also define spectral Hecke algebras at non-Taylor-Wiles primes, but they are not relevant for our global applications, just as the derived Hecke algebras at non-Taylor-Wiles primes are not relevant in \cite{V}.) 


We construct a \emph{co-algebra} structure on $\Cal{S}_q^{\Hk}$. This co-algebra structure does not descend to homotopy groups. (An analogous phenomenon is familiar in homology theory, where coproducts on chains may not descend to coproducts on cohomology because ``the K\"{u}nneth theorem points the wrong way''.)  However, it does descend after tensoring with a ring $\Lambda$ in which $q \equiv 1$. For such $\Lambda$ we get a coproduct on $\pi_* (\Cal{S}_q^{\Hk} \dotimes_{\Cal{O}} \Lambda)$, and then an algebra structure on the dualized (over $\Lambda$) homotopy groups $\pi_* (\Cal{S}_q^{\Hk} \dotimes_{\Cal{O}} \Lambda)^*$, since these homotopy groups are free over $\Lambda$.

We construct an isomorphism between this graded algebra and the local \emph{derived Hecke algebra} (in the sense of \cite{V}) with coefficients in $\Lambda$, denoted $\Cal{H}_q(\Lambda)$: 
\begin{equation}\label{eq: local hecke isom}
\pi_* (\Cal{S}_q^{\Hk} \dotimes_{\Cal{O}} \Lambda )^* \xrightarrow{\sim}\Cal{H}_q(\Lambda) .
\end{equation}
This is an arithmetic analogue of (a Koszul dual form of) the \emph{derived Geometric Satake equivalence} conjectured by Drinfeld, and proved by Bezrukavnikov-Finkelberg \cite{BF08}. 

We also construct a natural \emph{co-action} of $\Cal{S}_q^{\Hk}$ on the derived Galois deformation ring $\Cal{R}_S$. Again, this descends to homotopy groups after tensoring with $\Lambda$, and this leads to an algebra action of the $\Lambda$-dualized homotopy groups $\pi_* (\Cal{S}_q^{\Hk}\dotimes_{\Cal{O}} \Lambda)^*$ on the $\Lambda$-dualized homotopy groups $\pi_* (\Cal{R}_S\dotimes_{\Cal{O}} \Lambda)^*$. 

Now fix $\Lambda$ a finite quotient of $\Cal{O}$. We show (Theorem \ref{thm: compatibilty}) that for all $q$ which are congruent to $1$ modulo a sufficiently  high power of $p$ (depending on $\Lambda$)\footnote{This restriction comes from a similar such assumption in the statement of Venkatesh's Reciprocity Law \cite[Theorem 8.5]{V}. We expect both statements to be true without it.}, this action is intertwined with the action of $\Cal{H}_q(\Lambda)$ on $\wt{\mbb{T}}_{\mf{m}} \otimes_{\Cal{O}} \Lambda$ (via the homomorphism from local to global Hecke algebra) via the identifications \eqref{eq: local hecke isom} and \eqref{eq: global langlands isom}.
\[
\begin{tikzcd}[row sep = tiny]
\pi_* (\Cal{S}_q^{\Hk}\dotimes_{\Cal{O}} \Lambda)^* \ar[r, "\sim", "\eqref{eq: local hecke isom}"'] &  \Cal{H}_q(\Lambda) \\
\acts & \acts \\
\pi_* (\Cal{R}_S\dotimes_{\Cal{O}} \Lambda)^* \ar[r, "\sim", "\eqref{eq: global langlands isom}"'] & \wt{\mbb{T}}_{\mf{m}} \otimes_{\Cal{O}} \Lambda
\end{tikzcd}
\] 
We call this property ``derived local-global compatibility''; it bears a striking analogy to the strong Hecke compatibility in the Geometric Langlands Conjecture \cite[\S 4.7.4]{Ga15}. 

\begin{remark} The usual local-global compatibility at unramified places is essentially equivalent to saying that actions of the ``underived (i.e. degree $0$) parts'' $\pi_0 (\Cal{S}_q^{\Hk}\dotimes_{\Cal{O}} \Lambda)^* \xrightarrow{\sim}  \Cal{H}_q(\Lambda)^0 $ are intertwined. Of course, we are assuming this to begin with, and our Theorem really amounts to the assertion that the action of the ``derived parts'' then also match. 
\end{remark}



\subsection{Guide to the paper}

In \S \ref{sec: geometric Langlands} we summarize relevant aspects of Geometric Langlands theory. This is mainly for motivational purposes, and is logically independent of the paper. The reader may certainly skip it, but for our part we find the analogy with Geometric Langlands quite enlightening, and it was a helpful guide for developing this paper. 

In \S \ref{sec: spectral Hecke algebra} we define the spectral Hecke algebra $\Cal{S}_q^{\Hk}$, and study some of its basic invariants: homotopy groups, cotangent complex, and Andr\'{e}-Quillen (co)homology. 

In \S \ref{sec: spectral action} we construct the co-algebra structure on $\Cal{S}_q^{\Hk}$, and the co-action on the derived deformation ring of \cite{GV}. It is somewhat curious that we arrive at co-algebras and co-actions; \S \ref{ssec: analogy} discusses some (very loose) philosophical reasons why this happens in terms of the analogy to Geometric Langlands. 

In \S \ref{sec: comparison} we compare $\Cal{S}_q^{\Hk}$ to the local derived Hecke algebra studied in \cite{V}. This allows us to formulate ``derived local-global compatibility'', whose statement and proof occupy \S \ref{sec: derived local-global}.

\subsection{Acknowledgements}
The ideas here were conceived jointly with Akshay Venkatesh, although he declined to be named an author. We thank Matt Emerton, Soren Galatius, Dennis Gaitsgory, Akhil Mathew, and Xinwen Zhu for conversations related to this work.


\section{The spectral Hecke stack in Geometric Langlands}\label{sec: geometric Langlands} 

In this section we briefly explain the role of the spectral Hecke stack in Geometric Langlands, summarizing parts of \cite[\S 4]{Ga15}, \cite[\S 12]{AG15}. This is purely for motivational purposes, and has no logical impact on any of the later sections, so we keep our discussion informal. 

\subsection{The Geometric Langlands Conjecture} Let $X$ be a smooth, connected, projective curve over $\CC$ and $G$ be a reductive group over $\CC$. Associated to $X$ we have $\Bun_G$, the moduli stack of $G$-bundles on $X$, and $\LocSys_{\wh{G}}$, the moduli stack of $\wh{G}$-local systems on $X$. 

The \emph{Geometric Langlands Conjecture}, as formulated in \cite[Conjecture 1.1.6]{AG15}, predicts an equivalence of categories:
\begin{equation}\label{eq: GLC}
\LL_G \co \IndCoh_{\Nilp}(\LocSys_{\wh{G}})  \xrightarrow{\sim} \Dmod(\Bun_G) .
\end{equation}
Furthermore, it demands that this equivalence satisfies certain compatibility properties. The one which is relevant to this paper is the categorical analogue of the requirement that ``Hecke eigenvalues = Frobenius eigenvalues'' in the classical Langlands correspondence. (Note that the conjecture \eqref{eq: GLC} corresponds to ``everywhere unramified'' representations, so this is the only form of local-global compatibility needed.)

\subsection{Automorphic Hecke stack} 

We first explain the Hecke stack on the automorphic side. Let $x\in X(\CC)$, $\Cal{O}_x$ be the completed local ring of $X$ at $x$, and $F_x$ be its fraction field. We denote by $D_x := \Spec \Cal{O}_x$ the ``disk around $x$'', and $D_x^* := \Spec F_x$ the ``punctured disk around $x$''. 

The \emph{local Hecke stack} (at $x$) parametrizes ``two $G$-bundles on $D_x$, together with an isomorphism of their restrictions to $D_x^*$''. Any $G$-bundle on $D_x$ is trivial, and after choosing trivializations such an isomorphism is given by an element of $G(F_x)$. Hence it admits the presentation
\[
\Hecke(G, \aut)^{\loc}_x := \scr{L}^+G  \bs  \scr{L} G  / \scr{L}^+G,
\]
where $\scr{L}^+G$ is the \emph{arc space} of $G$ (a pro-algebraic group over $\CC$ whose $\CC$-points are $G(\Cal{O}_x)$), and $\scr{L}G$ is the loop group of $G$ (a group ind-scheme whose $\CC$-points are $G(F_x)$). The quotient is understood as a prestack, but what really matters is its category of sheaves, which can be understood more classically in terms of the presentation $\Dmod(\Hecke(G, \aut)_x^{\loc}) = \Dmod_{G(\Cal{O})}( \Gr_G)$. 


We denote a point of $\Hecke(G, \aut)^{\loc}(S)$ by $(\Cal{E} \dashrightarrow \Cal{E}')$, where $\Cal{E}$ and $\Cal{E}'$ are $G$-bundles on ``the disk around $x$'' (in the sense of $S$-points). We have a diagram 
\[
\begin{tikzcd}
& \Hecke(G, \aut)^{\loc}_x \ar[dl, "h^{\leftarrow}"'] \ar[dr, "h^\rightarrow"] \\
\Bun_{G, D_x}  & & \Bun_{G, D_x}
\end{tikzcd}
\]
where $h^{\leftarrow}(\Cal{E}, \Cal{E}', \Cal{E}|_{D_x^*}  \dashrightarrow \Cal{E}'|_{D_x^*}) = \Cal{E}$ and $h^{\rightarrow}(\Cal{E}, \Cal{E}', \Cal{E}|_{D_x^*}  \dashrightarrow \Cal{E}'|_{D_x^*}) = \Cal{E}'$. 

Restriction of bundles induces a map $\Bun_G \rightarrow \Bun_{G,D_x}$ for any $x$, and by the Beauville-Laszlo(-Drinfeld-Simpson) Theorem \cite{DS95} both squares in the commutative diagram below are cartesian. 
\[
\begin{tikzcd}
& \Hecke(G,\aut)_x^{\glob} \ar[dd] \ar[dl, "h^{\leftarrow}"']   \ar[dr, "h^\rightarrow"] \\
\Bun_G \ar[dd] & & \Bun_G \ar[dd] \\
& \Hecke(G,\aut)_x^{\loc} \ar[dl, "h^{\leftarrow}"'] \ar[dr, "h^\rightarrow"] \\
\Bun_{G, D_x}  & & \Bun_{G, D_x}
\end{tikzcd}
\]
This induces an action of $\Dmod(\Hecke(G,\aut)^{\glob}_x)$ on $\Dmod(\Bun_{G})$, by convolution: $\Cal{K} \in \Dmod(\Hecke(G,\aut)^{\glob}_x )$ acts on $\Cal{F} \in \Dmod(\Bun_{G})$ as 
\[
\Cal{F} \mapsto  h^{\leftarrow}_*( \Cal{K} \stackrel{!}\otimes (h^{\rightarrow})^! \Cal{F} ). 
\]
Composing this action with the pullback $\Dmod(\Hecke(G,\aut)_x^{\loc}) \rightarrow \Dmod(\Hecke(G,\aut)_x^{\glob} )$ induces an action of $\Dmod(\Hecke(G,\aut)_x^{\loc})$ on $D(\Bun_G)$, which is the analogue the action of classical Hecke operators at a place $x$ on the space of automorphic functions.

\begin{remark}\label{rem: hecke ran}
We can assemble the $\Hecke(\wh{G},\aut)_x^{\mrm{loc}}$-action, for varying $x$, into an action of $\Hecke(\wh{G}, \aut)^{\mrm{loc}}_{\Ran(X)}$ where the \emph{Ran space} $\Ran(X)$ parametrizes finite subsets of $X$ (see \cite[\S 4]{Ga15} for a concise discussion of this formalism). This is the analogue of assembling the local spherical Hecke algebras $H(G(\Z_p) \bs G(\Q_p) / G(\Z_p))$, as $p$ varies, into the adelic Hecke algebra. 
\end{remark}

\subsection{Spectral Hecke stack}\label{ssec: GL spectral Hecke stack}
We now formulate the analogue of $ \Hecke(G, \aut)_x^{\mrm{loc}}$, and its action, on the spectral side. Informally, this should parametrize ``two $\wh{G}$-local systems on a $D_x$, together with an isomorphism of their restrictions $D_x^*$'', meaning the fibered product of the diagram 
\[
\begin{tikzcd}
&  \LocSys_{\wh{G}, D_x} 
 \ar[d]  \\
 \LocSys_{\wh{G}, D_x} \ar[r] &   \LocSys_{\wh{G}, D_x^*}
 \end{tikzcd}
\]
where $\LocSys_{\wh{G}, D_x} $ is the space of $\wh{G}$-local systems on $D_x$, and $\LocSys_{\wh{G}, D_x^*}$ is the space of $\wh{G}$-local systems on $D_x^*$. Let's unwind what these objects are explicitly. 
\begin{itemize}
\item A $\wh{G}$-local system on $D_x$ is equivalent to the datum of a $\wh{G}$-torsor on $x$, which is necessarily trivial with automorphism group $\wh{G}$. Hence the space of such is $\LocSys_{\wh{G}, D_x}  = B\wh{G} := [\msf{pt}/\wh{G}]$. 

\item The formal neighborhood of the trivial local system in $\LocSys_{\wh{G}, D_x^*}$ is $\mf{g}/(\wh{G}, \Ad)$. This is easy to see for Betti local systems (representations of $\pi_1$), although our discussion has really been for de Rham local systems (vector bundles with connection). In the Betti case, a $\wh{G}$-local system on $D_x^*$ is specified by the monodromy, which is an element of $\wh{G}$ up to conjugation, and the formal neighborhood of the identity is isomorphic to $\mf{g}/(\wh{G}, \Ad)$ by the logarithm; the formal neighborhood of the trivial local system happens to be the same in the de Rham case.
\end{itemize}

Since a $\wh{G}$-local system on $D_x^*$ coming by restriction from one on $D_x$ is necessarily trivial, the map $\LocSys_{\wh{G}, D_x} \rightarrow \LocSys_{\wh{G}, D_x^*} $ sends $\msf{pt}$ to $0 \in \wh{\mf{g}}$. Clearly this fibered product is only interesting if we form it in a derived way. We define the \emph{local spectral Hecke stack} $\Hecke(\wh{G}, \mrm{spec})^{\mrm{loc}}_x $ to be the \emph{derived} fibered product
\[
\begin{tikzcd}
\Hecke(\wh{G}, \mrm{spec})^{\mrm{loc}}_x \ar[r, "h^{\rightarrow}"] \ar[d, "h^{\leftarrow}"] &  B\wh{G} \ar[d] \\
B\wh{G} \ar[r] &  \wh{\mf{g}}/(\wh{G}, \Ad)
\end{tikzcd}
\]

\subsubsection{Categories of sheaves} We are interested in certain categories of sheaves on $\Hecke(\wh{G}, \mrm{spec})^{\mrm{loc}}_x$. As was pointed out in \cite{AG15}, the singularities of $\LocSys_{\wh{G}}$ create some delicate issues in defining suitable categories of sheaves. The ``correct'' category to work with is $\IndCoh_{\Nilp}(\Hecke(\wh{G}, \mrm{spec})^{\mrm{loc}}_x)$, which contains $\QCoh(\Hecke(\wh{G}, \mrm{spec})^{\mrm{loc}}_x)$ as the full subcategory consisting of sheaves with 0 singular support. The nilpotent singular support has some connection with Arthur parameters, and it would be interesting to precisely understand the arithmetic analogue of this distinction. However we will eventually restrict our attention to tempered automorphic representations, and conjecturally the difference between these categories is invisible when acting on the ``tempered parts'' of \eqref{eq: GLC}, so we don't expect this subtlety to be meaningful for the purposes of this paper. 

\subsubsection{Monoidal structure} In general, a space of the form $\Cal{X} \times_{\Cal{V}} \Cal{X}$ has the structure of a groupoid over $\Cal{X}$, with the composition map 
\[
(\Cal{X} \times_{\Cal{V}} \Cal{X} ) \times_{\Cal{X}} (\Cal{X} \times_{\Cal{V}} \Cal{X})
\]
given by ``$(x_1, x_2), (x_2, x_3) \mapsto (x_1,x_3)$'' (cf. \S \ref{sec: spectral action} for more explanation). Applied to $\Hecke(\wh{G}, \mrm{spec})^{\mrm{loc}} $, we get a monoidal structure on $\IndCoh(\Hecke(\wh{G}, \mrm{spec})^{\mrm{loc}}_x )$, where we use $!$-pullback and $*$-pushforward (which preserves $\IndCoh_{\Nilp}$ and $\QCoh$). With this structure, the functor 
\[
\Rep(\wh{G}) = \QCoh(B\wh{G}) \rightarrow \IndCoh(\Hecke(\wh{G}, \mrm{spec})_x^{\mrm{loc}} ),
\]
given by pushforward across the diagonal map $\pt/\wh{G} \rightarrow \Hecke(\wh{G}, \mrm{spec})^{\mrm{loc}}_x$, is monoidal (with respect to the usual tensor product on $\Rep(\wh{G})$).

\subsection{Spectral Hecke action on local systems} \label{ssec: GL spectral action}

There is a map $\LocSys_{\wh{G}} \rightarrow \LocSys_{\wh{G}, D_x}$ given by restriction of local systems, and by \cite[eqn. (10.13)]{AG15} we have a presentation of $\LocSys_{\wh{G}}$ as the derived fibered product
\[
\begin{tikzcd}
\LocSys_{\wh{G}} \ar[r] \ar[d] & \LocSys_{\wh{G}}^{\mrm{R.S.},x} \ar[d] \\
B \wh{G} \ar[r] & \wh{\mf{g}}/\wh{G}
\end{tikzcd}
\]
where $ \LocSys_{\wh{G}}^{\mrm{R.S.},x} $ is the moduli stack of ``local systems with (at most) a simple pole at $x$''. (The arithmetic analogue of this cartesian square appears in \eqref{eq: adding ramification 1}.)

As explained in \cite[eqn. (12.11)]{AG15}, this induces a commutative diagram with all squares cartesian
\[
\begin{tikzcd}
&  \Hecke(\wh{G}, \mrm{spec})_x^{\mrm{glob}} \ar[dl, "h^{\leftarrow}"'] \ar[dr, "h^\rightarrow"] \ar[dd] \\ 
\LocSys_{\wh{G}} \ar[dd]  & & \LocSys_{\wh{G}} \ar[dd]  \\
& \Hecke(\wh{G}, \mrm{spec})_x^{\mrm{loc}} \ar[dl, "h^{\leftarrow}"'] \ar[dr, "h^\rightarrow"] \\
\LocSys_{\wh{G}, D_x} & & \LocSys_{\wh{G}, D_x}
\end{tikzcd}
\]
Hence one has an action of $\IndCoh_{\Nilp}(\Hecke(\wh{G},{\mrm{spec}}) )_x^{\mrm{glob}}$ on $\IndCoh_{\Nilp}(\LocSys_{\wh{G}})$ by convolution: $\Cal{K} \in \IndCoh_{\Nilp}(\Hecke(\wh{G}, \mrm{spec})_x^{\glob})$ acts on $\Cal{F} \in \IndCoh_{\Nilp}(\LocSys_{\wh{G}} )$ as 
\[
\Cal{F} \mapsto  h^{\leftarrow}_*( \Cal{K} \stackrel{!}\otimes (h^{\rightarrow})^! \Cal{F} ). 
\]
This induces an action of $\IndCoh_{\Nilp}(\Hecke(\wh{G},{\mrm{spec}}) ^{\mrm{loc}}_x)$ by composing with the pullback 
\[
\IndCoh_{\Nilp}(\Hecke(\wh{G},{\mrm{spec}}) ^{\mrm{loc}}_x) \rightarrow \IndCoh_{\Nilp}(\Hecke(\wh{G},{\mrm{spec}}) ^{\mrm{glob}}_x).
\]

\begin{remark} Parallel to Remark \ref{rem: hecke ran}, we can assemble the action of $\Hecke(\wh{G}, \mrm{spec})_x^{\mrm{loc}}$ into an action of $\Hecke(\wh{G}, \mrm{spec})^{\mrm{loc}}_{\Ran(X)}$ on $\IndCoh_{\Nilp}(\LocSys_{\wh{G}})$. 
\end{remark}

\subsection{Local-global compatibility}\label{ssec: GL local-global}

The \emph{derived Geometric Satake equivalence} of Ginzburg and Bezrukavnikov-Finkelberg \cite{BF08} induces by Koszul duality a monoidal equivalence \cite[\S 12.1.1]{AG15}
\[
\Sat \co \IndCoh_{\Nilp}(\Hecke(\wh{G},\mrm{spec})_x^{\loc}) \xrightarrow{\sim} \Dmod(\Hecke(G, \aut)_x^{\loc}).
\]

The ``Hecke compatibility'' aspect of the Geometric Langlands Conjecture demands that the equivalence $\LL_G$ from \eqref{eq: GLC} intertwines the automorphic and spectral Hecke actions through the Satake functor \cite[Conjecture 12.7.6]{AG15}:
\[
\begin{tikzcd}[row sep = tiny]
\IndCoh_{\Nilp}(\Hecke(\wh{G}, \mrm{spec})^{\loc}_{\Ran(X)} )  \ar[r, "\Sat", "\sim"'] & \Dmod(\Hecke(G, \aut)_{\Ran(X)}^{\loc}) \\
\acts & \acts \\
\IndCoh_{\Nilp}(\LocSys_{\wh{G}})   \ar[r, "\LL_G", "\sim"'] &  \Dmod(\Bun_G) 
\end{tikzcd}
\]

\section{The spectral Hecke algebra in arithmetic}\label{sec: spectral Hecke algebra}

We now introduce an arithmetic analogue of the spectral Hecke stack. 

\subsection{Motivation}

The arithmetic version of $D_x$ should be $\Spec \Z_q$ and the arithmetic version of $D_x^*$ should be $\Spec \Q_q$. So in the arithmetic case, we roughly propose to replace
\begin{align*}
\LocSys_{\wh{G}, D_x} & \rightsquigarrow \LocSys_{\wh{G}, \Z_q}, \\
\LocSys_{\wh{G}, D_x^*} & \rightsquigarrow \LocSys_{\wh{G}, \Q_q}.
\end{align*} 
Here $\LocSys_{\wh{G}, \Z_q}$ should be a moduli space of representation of $\pi_1(\Spec \Z_q)$, and $\LocSys_{\wh{G}, \Q_q}$ should be a moduli space of representations of $\pi_1(\Spec \Q_q) \cong \Gal(\ol{\Q}_q/\Q_q)$. 

We would then be interested in the derived fibered product
\[
\Hecke(\wh{G}, \mrm{spec})_{q} :=\LocSys_{\wh{G}, \Z_q} \times_{\LocSys_{\wh{G}, \Q_q}} \LocSys_{\wh{G}, \Z_q}.
\]
This is roughly the object that we will study, but some technical issues need to be addressed. One is the definition of the spaces ``$\LocSys_{\wh{G}, \Q_q}$'' and ``$\LocSys_{\wh{G}, \Z_q}$'', for which candidates are constructed \cite[\S 3]{Zhu20}, which lead to a good candidate for a spectral Hecke algebra, as in \cite[Conjecture 4.2.1]{Zhu20}.

Our approach will be different. Firstly, the functors $\LocSys_{\wh{G}, \Z_q} $ and $\LocSys_{\wh{G}, \Q_q}$ are not representable in general, so we need to introduce framings in order to  work with rings. For our present applications to studying the action on deformation spaces of global Galois representations, we need to complete at a given residual representation. Hence for our present purposes we work instead with formal deformation rings. 


\subsection{Definition of the spectral Hecke algebra} 

\subsubsection{Some notation}
Following the notation in \cite[\S 7.4]{GV}, let $q$ be a prime (the notation reflects that it will eventually be a ``Taylor-Wiles prime''). 

Let $k$ be a finite field of characteristic $p \neq q$ and $\Cal{O} = W(k)$. Let $\wh{G}$ be an algebraic group over $\Cal{O}$ and $\ol{\rho}$ be a representation of $\pi_1(\Z_q)$ into $\wh{G}(k)$, which we view by inflation as an unramified representation of $\Gal(\ol{\Q}_q/\Q_q)$.

  We let $\Cal{F}_{\Z_q, \ol{\rho}}$ be the (derived) deformation functor of $\ol{\rho}$, i.e. the functor parametrized unramified $G_{\Q_q}$-deformations of $\ol{\rho}$, from \cite[\S 7.4]{GV}. (See \S \ref{ssec: global derived def ring} for a brief discussion of how to define this.) 
 We let $\Cal{F}_{\Q_q, \ol{\rho}}$ denote the deformation functor of $\ol{\rho}$ as a $G_{\Q_q}$-deformation (here, the deformations are allowed to become ramified). 
 These are functors from simplicial commutative rings to simplicial sets; they are certainly not representable in general. 


\subsubsection{Taylor-Wiles primes} We now assume that $q$ is a \emph{Taylor-Wiles prime} for $\ol{\rho}$ in the sense of \cite[\S 6.7]{GV}, i.e. 
\begin{itemize}
\item $\ol{\rho}$ is unramified at $q$, 
\item $q \equiv 1 \in k$, 
\item $\ol{\rho}(\Frob_q)$ is conjugate to a strongly regular element of $\wh{T}(k)$.
\end{itemize}
(We do \emph{not} impose the Selmer condition that is often associated with the phrase ``Taylor-Wiles prime''.) This implies that $\rho$ admits a lift 
\[
\begin{tikzcd}
\pi_1 (\Z_q) \ar[r, "\ol{\rho}^{\wh{T}}"] \ar[dr, "\ol{\rho}"'] & \wh{T}(k) \ar[d] \\
& \wh{G}(k) 
\end{tikzcd}
\]
which is determined by $\Frob_q^{\wh{T}} := \ol{\rho}(\Frob_q)^{\wh{T}} \in \wh{T}(k)$. Abusing notation, \emph{we regard this choice of lift as part of the datum of a Taylor-Wiles prime}. (Later, the comparison to the derived Hecke algebra shows that the action is independent of this choice in the only reasonable sense.)

\subsubsection{Framed deformation rings} Let $q$ be a Taylor-Wiles prime for $\ol{\rho}$; henceforth we suppress $\ol{\rho}$ from the notation. Following the notation of \cite[\S 7.4]{GV}, let $\Cal{F}_{\Z_q, \ol{\rho}}^{\wh{T}, \square}$ and $\Cal{F}_{\Q_q,\ol{\rho}}^{\wh{T}, \square}$, denote the unramified and full framed deformation functors into $\wh{T}$, respectively. (This depend on the choice of lift $\ol{\rho}(\Frob_q) \in \wh{T}(k)$, which is suppressed in our notation.) These are pro-representable by pro-rings $\Cal{S}_q^{\mrm{ur}}$ and $\Cal{S}_q$, respectively. One can think of these as being the usual (non-derived) framed deformation rings, as follows. 

Recall that we say a pro-ring $\Cal{R}$ is \emph{homotopy discrete} if $\Cal{R} \rightarrow \pi_0 (\Cal{R})$ induces a weak equivalence of the induced pro-represented functors \cite[Definition 7.4]{GV}. By \cite[Lemma 8.6]{GV}, the rings $\Cal{S}_q^{\mrm{ur}}$ and $\Cal{S}_q$ are homotopy discrete. For our purposes, this means that one can simply regard them as discrete (i.e. non-simplicial) pro-rings, and by forming inverse limits as complete local Noetherian rings \cite[Lemma 7.2]{GV}. These complete local Noetherian rings then pro-represent the usual classical framed deformation functors.

\begin{defn}
The \emph{spectral Hecke algebra (at $q$, completed at $\ol{\rho}$)} is 
\[
\Cal{S}_q^{\Hk} :=\Cal{S}_q^{\mrm{ur}} \ul{\otimes}_{\Cal{S}_q} \Cal{S}_q^{\mrm{ur}}
\]
where the tensor product is the ``derived tensor product'', regarded as a simplicial commutative ring (meaning the tensor product of $\Cal{S}_q^{\mrm{ur}}$ with a cofibrant replacement of $\Cal{S}_q^{\mrm{ur}}$ as a $\Cal{S}_q$-algebra). 

The corresponding functor pro-represented by $\Cal{S}_q^{\Hk} $ will be denoted $\Hecke(\wh{G}, \spec)_q^{\loc}$. (The somewhat ``ad hoc'' use of framings in this definition is eventually justified by \S \ref{ssec: action on deformation ring}).)
\end{defn}

\begin{remark}\label{remark: pedantry}
In the usual category of commutative rings, constructions such as tensor products are unique up to unique isomorphism. This will never be the case for constructions we consider in the category of simplicial commutative rings; instead we get constructions that are, informally speaking, ``unique up to a contractible space of isomorphisms''. One way to express this is to say ``unique up to unique isomorphism in the homotopy category'', but this is not very good. In \cite{GV}, authors choose to work with the notion of ``naturally weakly equivalent'', which means that the two functors are related by a finite ``zig-zag'' of natural weak equivalences \cite[Definition 2.10]{GV}. The language of $\infty$-categories could probably provide a cleaner solution. 

These expository issues do not affect any calculation at the level of homotopy groups, (co)tangent complexes, Andr\'{e}-Quillen (co)homology, etc. Our ``official'' policy is to follow the convention of \cite{GV}. For two simplicial commutative rings $R,S$ we write $R \approx S$ or $R \xrightarrow{\sim} S$ to indicate a weak equivalence between $R$ and $S$ in the usual model structure on simplicial commutative rings. 
\end{remark}

\begin{remark}
One can make a more general definition of a spectral Hecke algebra at primes $q \neq p$ which are not of Taylor-Wiles type, by simply considering the framed deformation functor for $G$. Among the primes $q$ different from $p$, we expect the resulting object to be most interesting when $q$ is Taylor-Wiles, analogously to what happens for the local derived Hecke algebra in \cite{V}. However, when $q=p$ there should be a much richer story, and we have little idea what to expect. The analogous derived Hecke algebra has been investigated by Ronchetti \cite{Ron}. 
\end{remark}

\subsection{Explication in the Taylor-Wiles case}

Let $\mathrm{S}_q = \pi_0(\Cal{S}_q)$ and $\mrm{S}_q^{\ur} =\pi_0 (\Cal{S}_q^{\ur})$. It is also convenient to introduce the notation $\mrm{S}_q^\circ$ be the (underived) framed deformation ring for the trivial representation $I_q \rightarrow T$, where $I_q \approx (\Z/q)^{\times}$ is the tame inertial subgroup of $\Gal(\ol{\Q}_q/\Q_q)^{\ab}$.  

For a finitely generated abelian group $\Gamma$, let $\Gamma_{(p)}$ denote the quotient of $\Gamma$ by all of its prime-to-$p$ torsion. Following the notation of \cite[Remark 8.7]{GV}, we write $T(\Q_q)^{\ur} := T(\Q_q)/T(\Z_q)$ and $T(\Q_q)^{\mrm{tame}}$ for the profinite completion of $T(\Q_q)/\ker(T(\Z_q) \rightarrow T(\F_q))$. The usual computation of the deformation space at a Taylor-Wiles prime \cite[Remark 8.7]{GV} shows that 
\begin{align*}
\mathrm{S}_q^{\ur} &= \text{completed group algebra of } T(\Q_q)^{\ur}_{(p)}, \\
\mathrm{S}_q^\circ &= \text{completed group algebra of } T(\F_q)_{(p)},\\
\mathrm{S}_q &= \text{completed group algebra of } T(\Q_q)^{\mrm{tame}}_{(p)}.
\end{align*}

We can write this explicitly in coordinates if we choose an isomorphism $T \approx \G_m^r$. Let $p^N$ be the highest power of $p$ dividing $q-1$, so 
\begin{align*}
T(\Q_q)^{\ur}_{(p)} &\approx \wh{\Z}^r, \\
T(\F_q)_{(p)} & \approx (\Z/p^N \Z )^r,\\
T(\Q_q)^{\mrm{tame}}_{(p)} & \approx \wh{\Z}^r \times (\Z/p^N \Z)^r.
\end{align*}
Then we have
\begin{align*}
\mathrm{S}_q^{\ur}&  \approx \Cal{O}[[X_1, \ldots, X_r]] ,\\
\mathrm{S}_q^\circ & \approx  \Cal{O}[[Y_1, \ldots, Y_r]]\langle (1+Y_i)^{p^N}-1 \rangle \xleftarrow{\sim} \Cal{O}[Y_1, \ldots, Y_r]\langle (1+Y_i)^{p^N}-1 \rangle , \\
\mathrm{S}_q & \approx \Cal{O}[[X_1, \ldots, X_r]][Y_1, \ldots, Y_r]/\langle (1+Y_i)^{p^N}-1 \rangle.
\end{align*}

Since $\Cal{S}_q$ and $\Cal{S}_q^{\ur}$ are already homotopy discrete, we can calculate ``the'' derived tensor product using $\mathrm{S}_q^{\ur}$ and $\mathrm{S}_q $: 
\[
\Cal{S}_q^{\ur} \dotimes_{\Cal{S}_q } \Cal{S}_q^{\ur} \xrightarrow{\sim} \mrm{S}_q^{\ur} \dotimes_{\mrm{S}_q} \mrm{S}_q^{\ur} \cong \mrm{S}_q^{\ur} \dotimes_{\Cal{O}} (\Cal{O} \dotimes_{\Cal{O}[T(\F_q)_{(p)}]} \Cal{O}),
\]
where the last isomorphism follows from the fact that $\mrm{S}_q^{\ur} $ is already free over $\Cal{O}$. Hence we find 
\begin{equation}\label{eq: SHA = group homology}
\Cal{S}_q^{\Hk} \approx \mrm{S}_q^{\ur} \otimes_{\Cal{O}} (\Cal{O} \dotimes_{\Cal{O}[T(\F_q)_{(p)}]} \Cal{O}).
\end{equation}

Denote $T_q := T(\F_q)_{(p)}$. The underlying simplicial $\Lambda$-module of $\Lambda \dotimes_{\Lambda[T_q]} \Lambda$ is exactly what is used to compute to compute the group homology of $T_q$: 
\[
\Tor^{\Lambda[T_q]}_* (\Lambda, \Lambda) = H_*(T_q; \Lambda). 
\]
Hence \eqref{eq: SHA = group homology} implies:

\begin{cor}\label{cor: homotopy groups of SHA} We have
\[
\pi_* (\mrm{S}_q^{\Hk} \dotimes_{\Cal{O}} \Lambda) \cong \mrm{S}_q^{\ur} \otimes_{\Cal{O}} H_*(T_q; \Lambda).
\]
\end{cor}

\subsection{The tangent complex}\label{ssec: tangent complex of Fib}

Let $\Lambda$ a coefficient ring of the form $\Cal{O}/p^m$ for some $m \geq 1$. Suppose we are given an unramified deformation 
\[
\rho_{\Lambda} \co \pi_1(\Z_q) \rightarrow \wh{G}(\Lambda).
\]
We may then consider the deformation functors $\Cal{F}_{\Z_q, \rho_{\Lambda}}$ and $\Cal{F}_{\Q_q, \rho_{\Lambda}}$ of $\rho_{\Lambda}$ on $\Lambda$-augmented Artinian rings. 

For any functor $\Cal{F}$ on $\Lambda$-augmented $\Lambda$-rings, equipped with a given $0$-simplex of $\Cal{F}(\Lambda)$, we may consider the \emph{tangent complex} $\mf{t} \Cal{F}$ in the sense of \cite[Proof of Lemma 15.1]{GV}. This has homotopy groups $\mf{t}_i(\Cal{F}) := \pi_{-i}(\mf{t} \Cal{F} ) $ being the homotopy classes of maps $\Cal{F}(\Lambda \oplus \Lambda[i])$ lying over the given $0$-simplex of $\Cal{F}(\Lambda)$. 

\begin{remark}
Note that by the strong regularity assumption on $\rho(\Frob_q)$, our initial choice of lift $\ol{\rho}^{\wh{T}}(\Frob_q) \in \wh{T}(k)$ induces a lifting 
\[
\begin{tikzcd}
\pi_1 (\Z_q) \ar[r, "\rho_{\Lambda}^{\wh{T}}"] \ar[dr, "\rho_{\Lambda}"'] & \wh{T}(\Lambda) \ar[d] \\
& \wh{G}(\Lambda) \\
\end{tikzcd}
\]
That is, we automatically get a lift $\rho_{\Lambda}^{\wh{T}} (\Frob_q) \in \wh{T}(\Lambda)$, without making any additional auxiliary choices.
\end{remark}



\subsubsection{A fibration sequence}\label{sssec: fibration sequence} Our fixed representation $\rho_{\Lambda}$ gives a basepoint 
\[
\Spec \Lambda \xrightarrow{\msf{pt}} \Cal{F}_{\Z_q, \rho_{\Lambda}} \rightarrow \Cal{F}_{\Q_q, \rho_{\Lambda}}.
\]
Let $\Fib_{q,\Lambda}$ denote the homotopy fiber of the map $\Cal{F}_{\Z_q, \rho_{\Lambda}} \rightarrow \Cal{F}_{\Q_q, \rho_{\Lambda}}$ over $\msf{pt}$:
\[
\Fib_{q,\Lambda} = \Spec \Lambda \times^h_{\Cal{F}_{\Q_q}, \rho_{\Lambda}} \Cal{F}_{\Z_q, \rho_{\Lambda}}.
\]

\begin{remark}
The lift $\rho_{\Lambda}^{\wh{T}}(\Frob_q) \in \wh{T}(\Lambda)$ induces, as in \cite[eqn. (8.2)]{GV}, a cartesian diagram with compatible basepoints: 
\[
\begin{tikzcd}
\Cal{F}_{\Z_q, \rho_{\Lambda}} \ar[r] \ar[d] &  \Cal{F}_{\Z_q, \rho_{\Lambda}}^{\wh{T}, \square} \ar[d] \\
\Cal{F}_{\Q_q, \rho_{\Lambda}} \ar[r] & \Cal{F}_{\Q_q, \rho_{\Lambda}}^{\wh{T}, \square}
\end{tikzcd}
\]
giving a natural weak equivalence
\begin{equation}\label{eq: w.e. for Fib_q} 
\Fib_{q,\Lambda} \xrightarrow{\sim} \Spec \Lambda \stackrel{h}\times_{\Cal{F}_{\Q_q,\rho_{\Lambda}}^{\wh{T}, \square}} \Cal{F}_{\Z_q, \rho_{\Lambda}}^{\wh{T}, \square}.
\end{equation}
\end{remark}

\subsubsection{} Then tangent complex preserves homotopy pullbacks \cite[Lemma 4.30(iv)]{GV}, giving us the long exact sequence of Andr\'{e}-Quillen cohomology with coefficients in $\Lambda$:
\begin{align*}
\ldots & \rightarrow \mf{t}_0( \Fib_{q,\Lambda}) \rightarrow \mf{t}_0 (\Cal{F}_{\Z_q, \rho_{\Lambda}} )\rightarrow \mf{t}_0(\Cal{F}_{\Q_q,\rho_{\Lambda}} )\\
&\rightarrow \mf{t}_1( \Fib_{q,\Lambda} )  \rightarrow \mf{t}_1  (\Cal{F}_{\Z_q,\rho_{\Lambda}} )  \rightarrow \mf{t}_1 (\Cal{F}_{\Q_q,\rho_{\Lambda}} )  \\
& \rightarrow  \mf{t}_2 ( \Fib_{q,\Lambda}) \rightarrow \mf{t}_2(  \Cal{F}_{\Z_q,\rho_{\Lambda}})  \rightarrow \mf{t}_2 (\Cal{F}_{\Q_q,\rho_{\Lambda}}) \rightarrow \ldots 
\end{align*}
 As in \cite[Example 5.6]{GV} and \cite[Lemma 15.1]{GV}, we have 
 \begin{align*}
\mf{t}_i (\Cal{F}_{\Z_q,\rho_{\Lambda}} )& = H^{i+1}(\Z_q; \Ad \rho_{\Lambda}),  \\
\mf{t}_i (\Cal{F}_{\Q_q,\rho_{\Lambda}} )&= H^{i+1}(\Q_q; \Ad \rho_{\Lambda}) .
\end{align*}
Splicing this in above, we get 
\begin{align*}
\ldots & \rightarrow \mf{t}_0 (\Fib_{q,\Lambda}) \rightarrow  H^{1}(\Z_q; \Ad \rho_{\Lambda})  \rightarrow H^{1}(\Q_q; \Ad \rho_{\Lambda}) \\
&\rightarrow \mf{t}_1(\Fib_{q,\Lambda})  \rightarrow \underbrace{H^2(\Z_q; \Ad \rho_{\Lambda})}_{=0}  \rightarrow H^2(\Q_q; \Ad \rho_{\Lambda})   \\
& \rightarrow \mf{t}_2(\Fib_{q,\Lambda} ) \rightarrow 0  \rightarrow 0  \rightarrow \ldots 
\end{align*}

\subsubsection{Calculation of $\mf{t}_0(\Fib_{q,\Lambda})$} Since $H^{1}(\Z_q; \Ad \rho_{\Lambda})  \hookrightarrow H^{1}(\Q_q; \Ad \rho_{\Lambda})$, we find that $\mf{t}_0( \Fib_{q,\Lambda})= 0$. 

\subsubsection{Calculation of $\mf{t}_1(\Fib_{q,\Lambda})$}\label{sssec: t_1(Fib)} The long exact sequence gives an isomorphism 
\begin{equation}\label{eq: t_1 Fib map}
\mf{t}_1( \Fib_{q,\Lambda}) \xrightarrow{\sim}   H^{1}(\Q_q; \Ad \rho_{\Lambda})/H^{1}(\Z_q; \Ad \rho_{\Lambda}) .
\end{equation}
This is the ``ramified part'' of the deformation space for $\rho$. The fact that $\ol{\rho}$ is unramified forces any such deformations to be tamely ramified. Then \cite[Lemma 8.3]{GV} shows that the deformation functor into $\wh{G}$ is weakly equivalent to the deformation functor into $\wh{T}$, and in particular:
\[
\frac{H^{1}(\Q_q; \Ad \rho_{\Lambda})}{H^{1}(\Z_q; \Ad \rho_{\Lambda})}  \cong \frac{H^1(\Q_q; \Lie(\wh{T}) \otimes \Lambda) }{H^1(\Z_q; \Lie(\wh{T}) \otimes \Lambda )} \cong \Hom(I_q, \Lie(\wh{T}) \otimes \Lambda),
\]
where $I_q$ is the tame inertial subgroup of $\Gal(\ol{\Q}_p/\Q_p)^{\ab}$. As $ \Lie(\wh{T}) = X_*(\wh{T}) \otimes \Cal{O}$, we have by class field theory 
\begin{align*}
\Hom(I_q, \Lie(\wh{T}) \otimes_{\Cal{O}} \Lambda) & \cong \Hom(\F_q^{\times}, X_*(\wh{T}) \otimes_{\Cal{O}} \Lambda) \\
&  \cong \Hom(\F_q^{\times} \otimes X_*({T}), \Lambda) \cong \Hom(T_q, \Lambda).
\end{align*}
Hence we conclude that 
\[
\mf{t}_1(\Fib_{q,\Lambda})  \cong H^1(T_q; \Lambda) \cong H_1(T_q; \Lambda)^{*},
\]
where $H_1(T_q; \Lambda)^* := \Hom_{\Lambda}(H_1(T_q; \Lambda), \Lambda)$.

\subsubsection{Calculation of $\mf{t}_2(\Fib_{q,\Lambda})$}\label{sssec: t_2(Fib)}

The long exact sequence immediately shows that $\mf{t}_2(\Fib_{q,\Lambda}) \cong H^2(\Q_q; \Ad \rho_\Lambda)$, but we want to write this in another way. Again by \cite[Lemma 8.3]{GV}, the map 
\[
H^2(\Q_q; \Lie(\wh{T}) \otimes \Lambda)  \rightarrow H^2(\Q_q; \Lie(\wh{G}) \otimes\Lambda)
\]
is an isomorphism.  By Tate local duality, 
\[
H^2(\Q_q; \Lie(\wh{T}) \otimes \Lambda)   \cong H^0(\Q_q; (\Lie(\wh{T}) \otimes {\Lambda})^{*} (1))^{*}
\]
where $^*$ denotes the Pontrjagin dual (i.e. dual over $\Lambda$, in our situation) and $(1)$ denotes the Tate twist. Let $p^m$ be the smallest power of $p$ which is is $0$ in $\Lambda$; our assumption implies $q \equiv 1 \pmod{p^m}$. Now, we have canonical identifications 
\begin{align*}
(\Lie(\wh{T}) \otimes_{\Cal{O}} \Lambda)^*(1)  & \xrightarrow{\sim} X^*(\wh{T})_{\Lambda} \otimes_{\Z/p^m \Z}  \mu_{p^m}  \\
& = X_*(T) \otimes_{\Cal{O}} \Lambda \otimes_{\Z/p^m \Z}  \mu_{p^m} \xrightarrow{\sim} T(\F_q)[p^m] \otimes_{\Z/p^m \Z} \Lambda.
\end{align*}
Hence we have constructed an isomorphism 
\begin{equation}\label{eq: t_2(Fib)}
H^2(\Q_q; \Lie(\wh{T}) \otimes_{\Cal{O}} \Lambda)  \cong \Hom(T(\F_q)[p^m], \Lambda)  \cong H_2(T_q; \Lambda)_{\mrm{prim}}^{*}
\end{equation}
where $H_2(T_q; \Lambda)_{\mrm{prim}}$ is the subspace of primitives elements in $H_2(T_q; \Lambda)$ with respect to the coproduct on $H_*(T_q, \Lambda)$ dual to the cup product on $H^*(T_q; \Lambda)$. In other words, $H_2(T_q; \Lambda)_{\mrm{prim}}$ is dual to the indecomposable quotient of $H^2(T_q; \Lambda)$. Non-canonically, if we choose $T_q \xrightarrow{\sim} (\F_q^{\times})^r$, then $ H_2(T_q; \Lambda)_{\mrm{prim}} \xrightarrow{\sim} H_2(\F_q^{\times}; \Lambda)^{\oplus r}$. 

\section{Co-action on the global derived Galois deformation ring}\label{sec: spectral action} 

\subsection{Analogies and metaphors}\label{ssec: analogy}

The ``categorical trace of Frobenius'' formalism \cite{Ga15}, \cite{GKNV} can be used to turn categorical statements into function-theoretic statements in a systematic way. The Galois deformation ring looks like the categorical trace of Frobenius on (the category of quasicoherent sheaves on) the formal completion of $\LocSys_{\wh{G}}$ at a point, and our spectral Hecke algebra looks like the trace of Frobenius on (the category of quasicoherent sheaves on) the formal completion of $\Hecke(\wh{G}, \spec)^{\loc}_x$ at the corresponding point. Therefore, trying to take the categorical trace of Frobenius of the action in \S \ref{ssec: GL spectral action} would lead one to expect an action of the spectral Hecke algebra on the global Galois deformation ring. 

Note however that by the discussion of \S \ref{ssec: GL spectral action}, the algebra structure for this action should not be for the multiplication of functions, which would be the trace of the monoidal structure given by tensor product on $\QCoh(\Hecke(\wh{G}, \spec)^{\loc}_x)$. In the context of the analogy between $\QCoh$ and functions, there is also a loose analogy between $\IndCoh$ and ``measures'' \cite[Preface \S 1.3]{GR17}, which suggests that we should instead be considering a ring structure that comes from a ``convolution of measures'' with respect to the map 
\begin{equation}\label{eq: hecke convolution}
\Hecke(\wh{G}, \spec)^{\loc}_q \times_{\LocSys_{\wh{G}, \Z_q}} \Hecke(\wh{G}, \spec)^{\loc}_q \rightarrow \Hecke(\wh{G}, \spec)^{\loc}_q.
\end{equation}
We don't see how to make sense of this formally, so we change the game: the diagram \eqref{eq: hecke convolution} induces a \emph{co-algebra} structure on rings of functions via pullback, which would be dual to a convolution product on measures if that actually existed. Therefore, we will define a \emph{co-action} of the spectral Hecke algebra on the global derived deformation ring. Then the local-global compatibility of \S \ref{ssec: GL local-global} suggests that this action should look dual to the action of the derived Hecke algebra on the cohomology of arithmetic groups. 

This seems to be justified by the global picture: \cite[\S 15]{GV} explains that (under favorable assumptions) the global derived Galois deformation ring and the global derived Hecke algebra are \emph{dual}, and act in a dual manner on the cohomology of arithmetic groups. In some sense our results give a local ``explanation'' for the appearance of this duality.

\subsection{Groupoids arising from Hecke-type constructions}\label{ssec: hecke groupoids}
In the hope of putting the forthcoming constructions in a broader context, we begin with a brief discussion of the underlying ``pattern'' of groupoids and groupoid actions arising from Hecke-type constructions. This subsection is somewhat motivational, and can safely be skipped. The point of presenting it is to clarify the relevant structure in an idealized situation, whereas we will later be studying a more homotopy-theoretic situation where the discussion would be muddled by concerns related to homotopy coherence. 

\subsubsection{Groupoid actions}
We recall the formalism of groupoid actions \cite[\href{https://stacks.math.columbia.edu/tag/0230}{Tag 0230}]{stacks-project}. A \emph{groupoid $\Cal{G}$} in a category $\Cal{C}$ (with fibered products) consists of the following data:
\begin{enumerate}
\item A pair of objects $\mrm{Arr}, \mrm{Ob} \in \Cal{C}$ with two maps (``source'' and ``target'') $s,t \co \mrm{Arr} \rightrightarrows \mrm{Ob}$.
\item (``Identity'') A map $e \co \mrm{Ob} \rightarrow \mrm{Arr}$. 
\item (``Inverse'') A map $i \co \mrm{Arr} \rightarrow \mrm{Arr}$. 
\item (``Composition'') A partially defined composition law 
\[
\mu \co \mrm{Arr} \times_{s, \mrm{Ob}, t} \mrm{Arr} \rightarrow \mrm{Arr}.
\]
\end{enumerate}
These must satisfy: associativity of $\mu$, an ``identity axiom'', and an ``inverse axiom''. In this situation we say that ``$\mrm{Arr}$ is a groupoid over $\mrm{Ob}$''. 

Let $\Cal{G} = (\mrm{Arr}, \mrm{Ob},s,t,e,i,\mu)$ be a groupoid in $\Cal{C}$, and $E \in \Cal{C}$ be an object. An \emph{action} of $\Cal{G}$ on $E$ is defined by the data of: 
\begin{enumerate}
\item a map $\pi \co E \rightarrow \mrm{Ob}$, and 
\item a map $ a \co  \mrm{Arr} \times_{s, \mrm{Ob}, \pi} E \rightarrow E$
\end{enumerate}
satisfying for all $g,h \in \mrm{Arr}$ and $e \in E$: $\pi(a(g,e)) = t(g)$ when this is defined, and $a((gh), e) = a(g, a(h,e))$ when this is defined. 

\subsubsection{Hecke-type constructions}
Now suppose $X,Y,Z$ are objects in a category $\Cal{C}$ that admits fibered products, and we have maps $f \co X \rightarrow Z$ and $g \co Y \rightarrow Z$. Then $Y \times_Z Y$ has the structure of a groupoid over $Y$, and that there is a natural action of $Y \times_Z Y$ on $Y \times_Z X$. 
 \begin{itemize}
\item The maps $s,t \co Y \times_Z Y \rightarrow Y$ are the obvious projections. 
\item The map $e: Y \rightarrow Y \times_Z Y$ is the diagonal. 
\item The map $i \co Y \times_Z Y \rightarrow Y \times_Z Y$ is the ``swap'' of the two factors of $Y$. 
\item The composition
\[
(Y \times_Z Y) \times_Y (Y \times_Z Y) \simeq Y \times_Z Y \times_Z Y \rightarrow Y \times_Z Y 
\]
is the projection to the outer two factors of $Y$ (alternatively interpreted, ``convolution over middle coordinate''). 
\end{itemize}

The action of $Y \times_Z Y$ on $Y \times_Z X$ specified by: 
\begin{enumerate}
\item $\pi \co Y \times_Z X \rightarrow Y$ is projection to the first factor. 
\item $a \co (Y \times_Z Y) \times_Y (Y \times_Z X) \rightarrow  Y \times_Z X$ is projection to the outer factors. 
\end{enumerate}

For psychological comfort, we give a few examples of how the preceding formalism is familiar in algebraic geometry. 

\begin{example}
Suppose that $Y \rightarrow Z$ is a $G$-torsor in schemes for a group scheme $G / Z$. Then $ Y \times_Z Y \cong G \times Y$, and the $Y \times_Z Y$-action on $Y$ is equivalent to the given $G$-action on $Y$. 
\end{example}

\begin{example}
Suppose $Y \rightarrow Z$ is a faithfully flat map of schemes. Then $\Cal{G} := Y \times_Z Y$ is a groupoid over $Y$. If $\Cal{F}$ is a sheaf on $Y$, then a $\Cal{G}$-equivariant structure on $\Cal{F}$ is equivalent to the usual notion of descent datum for the cover $Y \rightarrow Z$, which induces an equivalence of categories $\QCoh(Z) \cong \QCoh_{\Cal{G}}(Y)$.
\end{example}

\subsection{The derived Galois deformation ring}\label{ssec: global derived def ring} We now review the setup of the Galatius-Venkatesh derived Galois deformation ring, in preparation for the definition of the co-action. Let $\wh{G}$ be a split adjoint group with trivial center over $\Cal{O} = W(k)$.


Suppose we are given a Galois representation $\ol{\rho} \co \Gal(\ol{\Q}/\Q) \rightarrow \wh{G}(k)$ satisfying the assumptions in \cite[Conjecture 6.1]{GV}: in particular, we suppose $\ol{\rho}$ is Fontaine-Laffaille at $p$ and has ``large image'', i.e. $\mrm{image}(\ol{\rho}) \supset \mrm{image}(\wh{G}^{\mrm{sc}}(k) \rightarrow \wh{G}(k))$. Let $S$ be a finite set of places of $\Q$, containing $p$ and the ramified places of $\ol{\rho}$. 

There is a derived Galois deformation functor $\Cal{F}_{\Z[1/S], \ol{\rho}}^{\mrm{crys}}$, which sends an Artinian SCR $A$ augmented over $k$ to 
\begin{quote}
``the space of representations of $\Gal(\ol{\Q}/\Q) \rightarrow \wh{G}(A)$ unramified outside $S$, and crystalline at $p$, which reduce to $\ol{\rho}$''.
\end{quote}
This is actually rather delicate to define precisely; we will sketch it below. Galatius-Venkatesh show that it is pro-representable, and we denote by $\Cal{R}_{S}$ a representing pro-ring (suppressing the dependence on $\ol{\rho}$). By \cite[Lemma 7.1]{GV}, $\pi_0 (\Cal{R}_{S})$ recovers the usual (underived) ring pro-representing the usual crystalline deformation functor of $\ol{\rho}$.

Now we briefly sketch the definition of $\Cal{F}_{\Z[1/S], \ol{\rho}}^{\mrm{crys}}$. First we define a version without the crystalline condition, denoted $\Cal{F}_{\Z[1/S], \ol{\rho}}$. To do this we view $\pi_1(\Z[1/S]) = \pi_1(X)$ where $X$ is the \'{e}tale homotopy type of $\Spec \Z[1/S]$ in the sense of Friedlander, which is a pro simplicial set; we write $X =  (X_{\alpha})$ for a presentation of $X$ as a pro-system of simplicial sets. \cite{GV} considers the derived deformation functor $\Cal{F}_{\Z[1/S], \ol{\rho}}$ whose value on an SCR $A$ is the simplicial set obtained by taking the homotopy fiber of map (of simplicial sets)
\[
\varinjlim_{\alpha} \Hom( X_{\alpha}, BG(A)) \rightarrow \varinjlim_{\alpha} \Hom( X_{\alpha}, BG(k))
\]
over the zero-simplex $\rho$ in the codomain. Here $BG(A)$ is defined by mapping a (cofibrant replacement of) the bar construction for $\Cal{O}_G$ to $A$ (and not, as one might na\"{i}vely guess, as $B(G(A))$). See \cite[\S 5 and \S 7.3]{GV} for the details. Restriction induces a map  $\Cal{F}_{\Z[1/S], \ol{\rho}}^{\mrm{crys}} \rightarrow \Cal{F}_{\Q_v, \ol{\rho}}$ where $\Cal{F}_{\Q_v, \ol{\rho}}$ is an analogous local deformation functor. Locally one defines a crystalline deformation functor $\Cal{F}_{\Q_p, \ol{\rho}}^{\crys}$ by imposing the crystalline condition on $\pi_0(A)$, and the global crystalline deformation functor $\Cal{F}_{\Z[1/S], \ol{\rho}}^{\mrm{crys}}$ is then obtained by taking the homotopy fibered product of 
$\Cal{F}_{\Z[1/S], \ol{\rho}}$ and $\Cal{F}_{\Q_p, \ol{\rho}}^{\crys}$ over $\Cal{F}_{\Q_p, \ol{\rho}}$; see \cite[\S 9]{GV} for the details.

\subsection{Hecke co-action on derived deformation rings}\label{ssec: action on deformation ring}

Now we will essentially explicate the construction of \S \ref{ssec: hecke groupoids} in the category of derived schemes. However, the preceding discussion needs to be modified because this is a homotopy-theoretic situation, e.g. fibered product needs to become homotopy fibered product, etc. We will just forget the axiomatic framework and explicitly give the constructions for simplicial commutative rings.

  Let $A,B,C$ be SCRs, and $C \rightarrow A$ and $C \rightarrow B$ be homomorphisms of SCRs. Assume that $C \rightarrow A$ and $C \rightarrow B$ are both cofibrations. Then we have the following structure on $B \otimes_C B$:
  \begin{itemize}
\item Homomorphisms $s,t \co B \rightrightarrows (B \otimes_C B)$ into the first and second factors. 
\item An augmentation $e \co B \otimes_C B \rightarrow B$ given by multiplication.
\item A ``swap'' $i \co B \otimes_C B \rightarrow B \otimes_C B$.
  \item A coproduct 
\begin{equation}\label{eq: comult}
B \otimes_C B \rightarrow (B \otimes_C B) \otimes_B (B  \otimes_C B).
\end{equation}
sending $b_1 \otimes b_2 \mapsto b_1 \otimes 1 \otimes b_2$. 
\end{itemize}

We also have a co-action of $B \otimes_C B$ on $B \otimes_C A$ as $B$-algebras, given by the map
\[
B \otimes_C A \rightarrow   (B \otimes_C B)  \otimes_B (B \otimes_C A)  
\]
sending $b \otimes a \mapsto (b \otimes 1 ) \otimes (1 \otimes a)$.

We let $\Cal{R}_S$ be the global deformation ring of $\rho$ discussed above in \S \ref{ssec: global derived def ring}, and $\Cal{R}_{Sq}$ the global deformation ring allowing additional ramification at $q$, i.e. the same construction but with $S$ replaced by $S \cup \{q\}$. By \cite[\S 8]{GV} we have 
\begin{equation}\label{eq: adding ramification 1}
\Cal{F}_{S, \ol{\rho}}^{\crys}  \xrightarrow{\sim} \Cal{F}_{S q,\ol{\rho}}^{\crys}  \times_{\Cal{F}_{\Q_q,\ol{\rho}}}^h \Cal{F}_{\Z_q,\ol{\rho}} \xrightarrow{\sim} \Cal{F}_{S q,\ol{\rho}}^{\crys}  \times_{\Cal{F}_{\Q_q,\ol{\rho}}^{\wh{T},\square}}^h \Cal{F}_{\Z_q,\ol{\rho}}^{\wh{T}, \square}.
\end{equation}
Note that the first equality expresses the intuition that the space of deformations ramified at $S$ can be obtained from the space of deformations ramified at $Sq$ by imposing a local unramifiedness condition at $q$. At the level of representing (pro-)rings, this means that
\begin{equation}\label{eq: adding ramification 2}
\Cal{R}_{Sq} \ul{\dotimes}_{\Cal{S}_q}\Cal{S}_q^{\mrm{ur}} \approx \Cal{R}_S.
\end{equation}
Now we apply the preceding discussion with $C = \Cal{S}_q$, $A$ a cofibrant replacement of $\Cal{R}_{Sq}$ as a $C$-algebra, and $B$ a cofibrant replacement of $\Cal{S}_q^{\mrm{ur}}$ as a $C$-algebra, getting in particular a co-multiplication (not a priori co-commutative) over $\Cal{S}_q^{\mrm{ur}}$,
\[
\Cal{S}_q^{\Hk} \rightarrow \Cal{S}_q^{\Hk} \dotimes_{\Cal{S}_q^{\mrm{ur}}} \Cal{S}_q^{\Hk}
\]
and a co-action over $\Cal{S}_q^{\mrm{ur}}$,
\[
 \Cal{R}_S \rightarrow  \Cal{R}_S \dotimes_{\Cal{S}_q^{\mrm{ur}}}  \Cal{S}_q^{\Hk}.
\]

\section{Comparison with the derived Hecke algebra}\label{sec: comparison}

We will now explain the local comparison between the derived Hecke algebra and the spectral Hecke algebra at Taylor-Wiles primes. This step is analogous to the role of the derived Geometric Satake equivalence in \S \ref{ssec: GL local-global}.

\subsection{The local derived Hecke algebra}

We briefly review the theory of derived Hecke algebra from \cite{V}.

We keep the setup of \S \ref{ssec: global derived def ring}, so in particular $G$ is a split reductive group $\Q$. Let $U \subset G(\Q_q)$ be a compact open subgroup. (For our purposes, we can take $U = K_q$ to be the maximal compact subgroup.)

Denoting $\Lambda[G(\Q_q)/U]$ for the compact-induction of the trivial representation from $U$ to $G(\Q_q)$, we can present the usual Hecke algebra for the pair $(G(\Q_q),U)$ as 
\[
H(G(\Q_q),U; \Lambda) := \Hom_{G(\Q_q)}(\Lambda[G(\Q_q)/U],\Lambda[G(\Q_q)/U]).
\]
This presentation suggests the following generalization. 

\begin{defn}
The \emph{derived Hecke algebra} for $(G(\Q_q),U)$ with coefficients in a ring $\Lambda$ is 
\[
\Cal{H}(G(\Q_q),U;\Lambda) := \Ext^*_{G(\Q_q)}(\Lambda[G(\Q_q)/U],\Lambda[G(\Q_q)/U]),
\]
where the $\Ext$ is formed in the category of smooth $G(\Q_q)$-representations. For $U= K_q$, we abbreviate $\Cal{H}_q(G(\Q_q);\Lambda) := \Cal{H}(G(\Q_q), U;\Lambda)$. 
\end{defn}


We next give a couple more concrete descriptions of the derived Hecke algebra \cite[\S 2]{V}.

\subsubsection{Function-theoretic description}\label{sssec: DHA function-theoretic}

Let $x,y \in G(\Q_q)/U$ and $G_{xy} \subset G$ be the stabilizer of the pair $(x,y)$. We can think of $\Cal{H}(G(\Q_q),U; \Lambda)$ as consisting of functions 
\[
  G(\Q_q)/U \times G(\Q_q)/U \ni (x,y) \mapsto h(x,y) \in H^*(G_{xy};\Lambda)
\]
satisfying the following constraints: 
\begin{enumerate}
\item The function $h$ is ``$G$-invariant'' on the left. More precisely, we have 
\[
[g]^* h(gx,gy) = h(x,y)
\]
where $[g]^* \co H^*(G_{gx,gy}; \Lambda) \rightarrow H^*(G_{x,y}; \Lambda)$ is pullback by $\Ad(g)$. 
\item The function $h$ has finite support modulo $G$. 
\end{enumerate}
The multiplication is given by a convolution formula, where one uses the cup product to define multiplication on the codomain, and restriction/inflation to shift cohomology classes to the correct groups \cite[eqn. (22)]{V}.

\subsubsection{Double coset description}\label{sssec: double coset description}

For $x \in G/U$, let $U_x = \Stab_U(x)$. Explicitly, if $x = g_xU$ then $U_x :=  U \cap \Ad(g_x) U$.

We can also describe $\Cal{H}(G(\Q_q),U;\Lambda)$ as functions 
\[
x \in U \bs G(\Q_q) / U \mapsto h(x) \in H^*(U_x; \Lambda)
\]
which are compactly supported, i.e. supported on finitely many double cosets. (However, it is harder to describe the multiplication in this presentation.) 


\subsubsection{The derived Hecke algebra of a torus}
Let $T$ be a split torus. Let's unravel the derived Hecke algebra of the torus $T(\Q_q)$, using now the double coset model. We set $T^{\circ}  = T(\Z_q)$ for its maximal compact subgroup. Since $T$ is abelian we simply have $T^{\circ}_x = T^{\circ}$ for all $x$. We have $T(\Q_q)/T^{\circ} \cong X_*(T)$. Identify
\begin{equation}\label{eq: satake embed}
X_*(T) = T(\Q_q) / T^{\circ} \hookrightarrow G/K_q
\end{equation}
by the map $X_*(T) \ni \chi  \mapsto \chi(\varpi_q) \in G/K_q$, where $\varpi_q$ is a uniformizer of $\Q_q$.

Next, writing $T_q := T(\F_q)_{(p)}$ as in \S \ref{ssec: tangent complex of Fib}, there is a canonical splitting $T_q \rightarrow T^{\circ}$ that splits the reduction map, and induces an isomorphism on cohomology (since we assume that $q$ is distinct from the residue characteristic $p$ of $\lambda$) 
\[
H^*(T_q; \Lambda) \xleftarrow{\sim} H^*(T^{\circ}; \Lambda).
\]
The upshot is that $\Cal{H}(T(\Q_q); \Lambda)$ simply consists of compactly supported functions
\[
X_*(T) \rightarrow H^*(T_q;\Lambda)
\]
with the multiplication given by convolution; in other words,
\[
\Cal{H}(T(\Q_q); \Lambda) \cong \Lambda[X^*(T)] \otimes_\Lambda H^*(T_q; \Lambda),
\]

\subsubsection{The derived Satake isomorphism}

We henceforth assume that $q \equiv 1 \in \Lambda$. Let $U = G(\Z_q)$ be a hyperspecial maximal compact subgroup of $G(\Q_q)$. We consider an analog of the classical Satake transform for the derived Hecke algebra $\Cal{H}_q(G(\Q_q);\Lambda)$, which takes the form 
\[
``\text{Derived Hecke algebra for $G$}  \xrightarrow{\sim} (\text{Derived Hecke algebra for maximal torus})^W."
\]

More precisely, let $T$ be a split maximal torus of $G$ such that $U \cap T(\Q_q)$ is the maximal compact subgroup $T(\Q_q)$. We define the \emph{derived Satake transform}
\begin{equation}\label{eq: satake transform}
  \Cal{H}(G(\Q_q);\Lambda) \rightarrow \Cal{H}(T(\Q_q); \Lambda)
\end{equation}
simply by \emph{restriction} (in the function-theoretic model \S \ref{sssec: DHA function-theoretic}) along the map $(T(\Q_q)/T^{\circ})^2 \rightarrow (G/K_q)^2 $ from \eqref{eq: satake embed}. In more detail, let $h \in \Cal{H}_q(G; \Lambda)$ be given by the function 
\[
 (G_v/K_v)^2 \ni (x,y) \mapsto h(x,y) \in H^*(G_{x,y};\Lambda).
\]
Then \eqref{eq: satake transform} takes $h$ to the composition 
\[
\begin{tikzcd}
(T(\Q_q)/T^{\circ} )^2  \ar[r, hook] & (G(\Q_q)/K_q)^2 \ar[r, "h"] & H^*(G_{x,y};\Lambda)  \ar[r, "\res"]  &  H^*(T_{x,y};\Lambda)
\end{tikzcd}
\] 

\begin{remark}It may be surprising that this is the right definition, since the analogous construction in characteristic 0, on the usual underived Hecke algebra,  is far from being the usual Satake transform. It is only because of our assumptions on the relation between the characteristics (namely, that $q \equiv 1 \in \Lambda$) that this ``na\"{i}ve'' definition turns out to be correct. 
\end{remark}

\begin{thm}[{\cite[Theorem 3.3]{V}}] \label{thm: derived sat isom}
Let $W$ be the Weyl group of $T$ in $G$. Under the assumptions of this section, the map \eqref{eq: satake transform} induces an isomorphism 
\[
\mrm{dSat}_q \co \Cal{H}(G(\Q_q);\Lambda) \xrightarrow{\sim} \Cal{H}(T(\Q_q); \Lambda)^W.
\]
\end{thm}

\begin{remark}
Technically \cite[Theorem 3.3]{V} is phrased only for $\Cal{O} = \Z_p$ and $\Lambda = \Z/p^m \Z$, but the more general version stated above follows immediately from that version by flat base change. 
\end{remark}

\subsection{The derived Hecke algebra vs. the spectral Hecke algebra}

\subsubsection{Localization of the derived Hecke algebra}\label{sssec: hecke localization}

Recall that the definition of Taylor-Wiles datum at $q$ includes a specification of $\Frob_q^{\wh{T}} \in \wh{T}(k)$. This datum is equivalent to that of a homomorphism of abelian groups 
\[
X_*(T) = X^*(\wh{T}) \rightarrow k^{\times},
\]
which is in turn equivalent to a character 
\[
\chi_{\Frob_q^{\wh{T}}} \co \Lambda[X_*(T)] \rightarrow k.
\]
Let $\mf{m}_{\chi}$ be the kernel of $\chi_{\Frob_q^T} $, which is a maximal ideal of $\Lambda[X_*(T)]$.

To compare the automorphic and Galois sides, we need to a Hecke eigensystem which is compatible with our representation. Recall that $G$ is a split semisimple simply connected group over $\Q$. Let $K$ be a level structure for $G(\A)$ and $Y(K)$ the associated locally symmetric space. Fix a pro-$p$ coefficient ring $E$ with an augmentation $E \surj k$, and let $\TT_{K}$ be the Hecke algebra (generated by Hecke operators at ``good primes'') acting on $C^*(Y(K), E)$ in the derived category (cf. \cite[\S 6.6]{GV}). Fix a Hecke eigensystem $\TT_K \surj k$, say with maximal ideal $\mf{m}$. This pulls back to a maximal ideal $\mf{m}_q$ of the local underived Hecke algebra $H(\G(\Q_q); \Lambda)$, and we let $\Cal{H}_q(\Lambda)$ be the completed local ring of $\Cal{H}(G(\Q_q); \Lambda)$ at $\mf{m}_\chi$. By combining Theorem \ref{thm: derived sat isom} with \cite[eqn. (147)]{V}, we find an isomorphism 
\begin{equation}\label{eq: localized DHA}
\Cal{H}_q(\Lambda) \xrightarrow{\sim} \Lambda[X_*(T)]_{\mf{m}_{\chi}} \otimes_\Lambda H^*(T_q; \Lambda)
\end{equation}
where on the right hand side, $\Lambda[X_*(T)]_{\mf{m}_{\chi}}$ denotes the completed local ring of $\Lambda[X_*(T)]$ at $\mf{m}_{\chi}$. 

\begin{remark}\label{rem: localize satake}
Said geometrically, we are localizing the finite map of schemes corresponding to Theorem \ref{thm: derived sat isom} at points where it is totally split by the strong regularity assumption, hence we obtain an isomorphism of completed local rings.
\end{remark}

\begin{defn}We denote $H_q(\Lambda) := \Lambda[X_*(T)]_{\mf{m}_{\chi}}$, the degree $0$ part of $\Cal{H}_q(\Lambda)$. 
\end{defn}

\subsubsection{Homotopy groups of derived tensor products}
Let $R$ be a simplicial commutative ring, and $A$ and $B$ be simplicial $R$-algebras.

Recall the spectral sequence for homotopy groups of a tensor product \cite[eqn. (5.2)]{Qui70}:
\begin{equation}\label{eq: htpy of tensor product}
E_{ij}^2 = \Tor^{\pi_*(R)}_i(\pi_*(A) , \pi_*(B))_j \implies \pi_{i+j}(A \dotimes_R B).
\end{equation}
Here the $j$-grading comes from the grading on $\pi_*(A), \pi_*(B)$ as modules over $\pi_*(R)$. 

In particular, we always have an edge map 
\[
\pi_*(A \dotimes_R B) \rightarrow \pi_*(A) \otimes_{\pi_*(R)} \pi_*(B). 
\]

\begin{example}\label{ex: degenerate edge map}
If $R$ happens to be homotopy discrete with $\mrm{R} := \pi_0(R)$, and $\pi_*(A)$ or $\pi_*(B)$ is flat over $\mrm{R}$, then \eqref{eq: htpy of tensor product} degenerates on $E^2$ and this edge map is an isomorphism: 
\begin{equation}
\pi_*(A \dotimes_{R} B) \xrightarrow{\sim} \pi_*(A) \otimes_{\mrm{R}} \pi_*(B).
\end{equation}
\end{example}

\subsubsection{Comparison with the spectral Hecke algebra}\label{sssec: local comparison}

In \S \ref{sec: spectral action} we equipped the spectral Hecke algebra $\Cal{S}_q^{\Hk}$ with a coproduct over $\Cal{S}_q^{\ur}$. By Corollary \ref{cor: homotopy groups of SHA} we have 
\[
\pi_*( \Cal{S}_q^{\Hk} \dotimes_{\Cal{O}} \Lambda) \cong \mrm{S}_q^{\ur} \otimes_{\Cal{O}} H_*(T_q; \Lambda).
\]
If $q \equiv 1 \in \Lambda$, then $H_*(T_q; \Lambda)$ is actually free over $\Lambda$. In this case $\pi_*( \Cal{S}_q^{\Hk} \otimes_{\Cal{O}} \Lambda)$ is free over $\Lambda$. Hence Example \ref{ex: degenerate edge map} applies in our case with $R = \mrm{S}_q^{\ur}$ and $A= B  = \Cal{S}_q^{\Hk}$, implying that 
\[
\pi_*(\Cal{S}_q^{\Hk} \dotimes_{\mrm{S}_q^{\ur}} \Cal{S}_q^{\Hk} \dotimes_{\Cal{O}} \Lambda) \xrightarrow{\sim} \pi_*(\Cal{S}_q^{\Hk} \dotimes_{\Cal{O}} \Lambda) \otimes_{\mrm{S}_q^{\ur} \otimes_{\Cal{O}} \Lambda} \pi_* (\Cal{S}_q^{\Hk} \dotimes_{\Cal{O}} \Lambda).
\]
Hence the coproduct on $\Cal{S}_q^{\Hk} \dotimes_{\Cal{O}} \Lambda $ induces a coproduct on $\pi_* (\Cal{S}_q^{\Hk} \dotimes_{\Cal{O}} \Lambda)$. 

To compare this to the derived Hecke algebra, we dualize. Define  
\[
\pi_*( \Cal{S}_q^{\Hk} \dotimes_{\Cal{O}} \Lambda)^\vee := \Hom_{\mrm{S}_q^{\ur} \otimes_{\Cal{O}} \Lambda}(\pi_*( \Cal{S}_q^{\Hk} \dotimes_{\Cal{O}} \Lambda), \mrm{S}_q^{\ur} \otimes_{\Cal{O}} \Lambda). 
\]
Since $\pi_*( \Cal{S}_q^{\Hk} \dotimes_{\Cal{O}} \Lambda)$ is free over $\mrm{S}_q^{\ur} \otimes_{\Cal{O}} \Lambda $, the $\mrm{S}_q^{\ur}$-co-algebra structure on $\pi_*(\Cal{S}_q^{\Hk} \dotimes_{\Cal{O}} \Lambda)$ induces a $\mrm{S}_q^{\ur}\dotimes_{\Cal{O}} \Lambda$-algebra structure on $\pi_*( \Cal{S}_q^{\Hk} \dotimes_{\Cal{O}} \Lambda)^\vee$. 

By \eqref{eq: SHA = group homology} we can also present 
\[
\Cal{S}_q^{\Hk} \dotimes_{\Cal{O}} \Lambda \xrightarrow{\sim} \mrm{S}_q^{\ur} \otimes_{\Cal{O}} (\Lambda \dotimes_{\Lambda[T_q]} \Lambda).
\]
This implies that 
\begin{align*}
\Hom_{\mrm{S}_q^{\ur}}(\pi_*( \Cal{S}_q^{\Hk} \dotimes_{\Cal{O}} \Lambda), \mrm{S}_q^{\ur}) & \cong \Hom_{\Lambda}(\pi_*(\Lambda \dotimes_{\Lambda[T_q]} \Lambda), \Lambda \otimes_{\Cal{O}} \mrm{S}_q^{\ur})  \\
&\cong \mrm{S}_q^{\ur} \otimes_{\Cal{O}} \Hom_{\Lambda}
(H_*(T_q; \Lambda), \Lambda).
\end{align*}  
Since $T(\F_q) \cong (\Z/(q-1) \Z)^r$ and $q \equiv 1 \in \Lambda$, the term $\Hom_{\Lambda}
(H_*(T_q; \Lambda), \Lambda)$ is canonically identified with the group cohomology $H_*(T_q; \Lambda)$, and the coproduct on homology dualizes to the usual cup product on cohomology, by the general relation between the coproduct on $\Tor$ and the Yoneda product on $\Ext$ \cite[p. 648]{E95}. Hence we have an identification of algebras 
\begin{equation}\label{eq: dual homotopy groups}
\pi_*( \Cal{S}_q^{\Hk} \dotimes_{\Cal{O}} \Lambda)^\vee \cong \mrm{S}_q^{\ur} \otimes_{\Cal{O}} H^*(T_q; \Lambda).
\end{equation}


The classical Satake isomorphism gives an identification 
\[
H(G(\Q_q), \Z[q^{\pm 1/2}]) \xrightarrow{\sim} R(\wh{G}) \otimes \Z[q^{\pm 1/2}],
\]
where $R(\wh{G})$ the latter is the representation ring of $\wh{G}$, i.e. the Grothendieck group of the category of finite-dimensional complex $\wh{G}$-representations, equipped with multiplication induced by tensor product. Since the assumption $q \equiv 1 \in \Lambda$ equips $\Lambda$ with a canonical square root of $q$, we get an isomorphism 
\begin{equation}\label{eq: classical satake isomorphism}
H_q := H(G(\Q_q), \Lambda) \xrightarrow{\sim} R_{\Lambda}(\wh{G}) := R(G)\otimes_{\Z} \Lambda.
\end{equation}
Hence we may also view $\mf{m}_{\chi}$ as a maximal ideal of $R_{\Lambda}(\wh{G})$, which we denote by the same name. We have a finite map $R(G)\otimes_{\Z} \Lambda \rightarrow \Lambda[X^*(\wh{T})]$, which induces an isomorphism between the completion of $R_{\Lambda}(\wh{G})$ at $\mf{m}_{\chi}$ and $\mrm{S}_q^{\ur} \otimes_{\Cal{O}} \Lambda$ for the same reason as in Remark \ref{rem: localize satake}. Composing this with \eqref{eq: classical satake isomorphism} gives an isomorphism
\begin{equation}\label{eq: localized Satake}
\mrm{S}_q^{\ur} \otimes_{\Cal{O}} \Lambda \xrightarrow{\sim}  H_q(\Lambda) .
\end{equation}

Combining \eqref{eq: localized Satake} with \eqref{eq: dual homotopy groups} and \eqref{eq: localized DHA}, we have constructed an isomorphism of graded rings
\begin{equation}\label{eq: local identifications}
\pi_*(\Cal{S}_q^{\Hk} \dotimes_{\Cal{O}} \Lambda)^{\vee}  \xrightarrow{\sim} \Cal{H}_q(G(\Q_q);\Lambda)
\end{equation}
extending \eqref{eq: localized Satake} on $\pi_0$. Informally, this says that the ``dual of the spectral Hecke coalgebra is the derived Hecke algebra'' on the level of homotopy groups.

\section{Derived local-global compatibility}\label{sec: derived local-global}

\subsection{The automorphic side}\label{ssec: automorphic side}
We now return to the global situation. Recall that $G$ was a split semisimple group over $\Q$.

\subsubsection{Cohomology of locally symmetric spaces} 
Let $Y(K)$ be the locally symmetric space associated with level structure $K$. Let $\TT_K$, $\chi$, and $\mf{m}$ be as in \S \ref{sssec: hecke localization}. Hence $\chi \co \TT_K \rightarrow E$ for some finite extension $E/\Z$. We pick a prime $p$ and a place $\mf{p}$ of $E$ above $p$, on which we'll shortly impose various conditions.

 \emph{We now assume that $\chi$ is a tempered Hecke eigensystem.} Then $H^*(Y(K); \CC)_{\mf{m}}$ is supported degrees $[j_0, j_0+\delta]$ where $\delta = \rank G(\R) - \rank K_{\infty}$, and $2j_0 + \delta = \dim Y(K)$, as established in \cite[III \S 5.1, VII Theorem 6.1]{BW00}, \cite[5.5]{B81}. We impose the assumptions of \cite[\S 13.1]{GV}, and pick a prime $p$ such that 
\begin{enumerate}
\item $H^*(Y(K); \Z)$ is $p$-torsion free.
\item $p> \# W$, where $W$ is the Weyl group of $G$. 
\item $\Cal{O} := E_{\mf{p}}$ is unramified over $\Z_p$.
\item (``no congruences) The map $(\TT_K)_{\mf{m}} \rightarrow \Cal{O}_{\mf{p}}$ induced by completing $\chi$ is an isomorphism. 
\item $H_*(Y(K); \Cal{O})_{\mf{m}}$ vanishes outside $[j_0, j_0+\delta]$. 
\end{enumerate}
(These assumptions should all be satisfied for all sufficiently large $p$.)

\subsubsection{Global derived Hecke algebra}\label{sssec: global DHA} For any open compact subgroup $U_q \subset G_q$, the local derived Hecke algebra $\Cal{H}(G(\Q_q),U_q; \Cal{O}/p^n\Cal{O})$ acts on the cohomology of a locally symmetric space with level structure at $q$ corresponding to $U_q$ (see \cite[\S 2.6]{V}).

We consider the action of the local derived Hecke algebra $\Cal{H}(G_q, U_q; \Cal{O}/p^n\Cal{O})$ for all $q \equiv 1 \pmod{p^n}$ such that $K$ is hyperspecial at $q$, and take $U_q$ to be a hyperspecial maximal compact subgroup. These actions generate an algebra
\[
\wt{\TT}_{K,n} \subset \End(H^*(Y(K); \Cal{O}/p^n\Cal{O})).
\]
Venkatesh defines the \emph{global derived Hecke algebra} to be the subalgebra $\wt{\TT}_K \subset \End(H^*(Y(K); \Cal{O}))$ consisting of endomorphisms of the form $\varprojlim t_n$ for $t_n \in \wt{\TT}_{K,n}$ \cite[\S 2.13]{V}. Note that endomorphisms  do not come from any particular local derived Hecke algebra, but are glued from such in a trancendental way.

\subsection{The Galois side}\label{ssec: galois side}

\subsubsection{Global derived deformation ring} 

Let $k = \Cal{O}/\mf{p}$ be the residue field of $\Cal{O}$, and let $S$ be a finite set of primes containing $p$ and the places at which $K$ is not hyperspecial. We sometimes identify $S$ with an integer which is the product of the primes it contains. Conjecturally, there should exist a global Galois representation 
\[
\ol{\rho} \co \Gal(\ol{\Q}/\Q) \rightarrow G(k)
\]
corresponding to $\mf{m}$, which enjoys the properties listed in \cite[Conjecture 6.1]{GV} and \cite[\S 13.1(8)]{GV}. We assume the existence of such a $\ol{\rho}$, which furthermore satisfies the assumptions of \cite[\S 10]{GV}. In particular, 
\begin{enumerate}
\item $\ol{\rho}$ is unramified outside $S$ and odd at $\infty$.
\item The residual representation into $\wh{G}(k)$ has ``big image''.
\item $\ol{\rho}$ is Fontaine-Laffaille above $p$, and has trivial deformation theory at the other primes in $S$. 
\item $\ol{\rho}$ enjoys local-global compatibility. 
\item $\ol{\rho}$ admits a lift
\[
\rho_{\Cal{O}} \co  \Gal(\ol{\Q}/\Q) \rightarrow \wh{G}(\Cal{O}).
\] 
\end{enumerate}

Let $\Cal{R}_S$ be the derived Galois deformation ring for $\rho$ from \S \ref{ssec: global derived def ring}. 

\subsubsection{Compatibility with the global derived Hecke algebra} 
We now discuss the relationship between $\Cal{R}_S$ and $(\wt{\TT}_K)_{\mf{m}}$. The traditional Taylor-Wiles method aims to prove an ``$R=\TT$'' theorem of the form $\mrm{R}_S \xrightarrow{\sim} (\TT_K)_{\mf{m}}$. However the derived versions $\Cal{R}_S$ and $(\wt{\TT}_K)_{\mf{m}}$ are not even the same type of object, the former being homologically graded and the latter being cohomologically graded. In contrast to the global derived Hecke algebra, which naturally acts by degree-increasing endomorphisms on the \emph{cohomology} $H^*(Y(K); \Z_p)$, $\pi_*( \Cal{R}_S )$ naturally acts by degree-increasing endomorphisms on the \emph{homology} $H_*(Y(K); \Z_p)$. 

To state the comparison between $\Cal{R}_S$ and $(\wt{\TT}_K)_{\mf{m}}$, we use the cap product (and the assumptions we are imposing, which force $H^*(Y(K);\Z_p)$ to be torsion-free), the derived Hecke algebra also acts in a degree-\emph{decreasing} manner on $H_*(Y(K);\Z_p)$.

\begin{defn}
For a module $M$ over $\Lambda = \Cal{O}$ or $\Cal{O}/p^m \Cal{O}$, we let $M^* := \Hom_{\Lambda}(M, \Lambda)$. (We will only apply this to free modules over $\Lambda$.) 

We define $\msf{V} = H_f^1(\Z[1/S]; \Ad^* \rho_{\Cal{O}}(1))^{*}$; this is a free module over $\Cal{O}$ of rank $\delta$ by \cite[Lemma 8.8]{V}. We denote $\msf{V}^* := \Hom_{\Cal{O}}(V, \Cal{O})$. (More generally, for a finite free module $M$ over a coefficient ring $\Lambda$ we will denote $M^* := \Hom_{\Lambda}(M, \Lambda)$.)
\end{defn}

It is shown in \cite[Theorem 8.5]{V} that, under our assumptions, the action of the local derived Hecke algebra on $H^*(Y(K); \Z_p)$ can be ``patched'' in the sense of Taylor-Wiles to an action of $\msf{V}$ on $H^*(Y(K); \Cal{O})_{\mf{m}}$. This induces an identification $\msf{V} \xrightarrow{\sim} \wt{\TT}_{\mf{m}}^1$ (the degree 1 part of the global derived Hecke algebra completed at $\mf{m}$), and \cite[Theorem 8.5]{V} shows moreover that $\msf{V}$ freely generates an exterior algebra in $ \End(H^*(Y(K); \Cal{O})_{\mf{m}})$, which coincides with $\wt{\TT}_{\mf{m}}$. In particular, we get an isomorphism 
\[
 \wedge^* \msf{V} \xrightarrow{\sim} \wt{\TT}_{\mf{m}} .
\]

On the other hand, \cite[\S 15]{GV} constructs an isomorphism (under our running assumptions)
\begin{equation}\label{eq: identify homotopy groups}
\pi_*( \Cal{R}_S ) \xrightarrow{\sim} \wedge^* (\msf{V}^*)
\end{equation}
and \cite[Theorem 14.1]{GV} constructs a natural action of $\pi_*( \Cal{R}_S )$ on $H_{j_0+*}(Y(K); \Cal{O})_{\mf{m}}$, realizing the latter as a free module of rank one over $\pi_*(\Cal{R}_S)\cong \wedge^* (\msf{V}^*)$. 

It is also established in \cite[Theorem 15.2]{GV} that these two actions are compatible in the natural way \cite[\S 15.2]{GV}. To articulate this precisely, we frame it more abstractly. Suppose $\msf{V}$ is a finite free $\Lambda$-module and $\msf{V}^*$ is its $\Lambda$-linear dual. If $M$ is a finite free $\Lambda$-module with actions of $\wedge^{*} \msf{V}$ and $\wedge^{*} \msf{V}^*$, we say that the two actions are \emph{compatible} if for all $v^* \in \msf{V}^*$ and $v \in \msf{V}$ and $m\in M$ we have 
\[
v \cdot v^* \cdot m + v^* \cdot v \cdot m  = \langle v, v^* \rangle \cdot m.
\]

\subsubsection{Hurewicz map}\label{sssec: pi_1 isom}

Let us describe the map $\pi_1(\Cal{R}_S) \xrightarrow{\sim} \msf{V}^*$ from \eqref{eq: identify homotopy groups}. It comes from a ``Hurewicz-like'' construction. 

Let $\Cal{R}$ be a simplicial commutative ring, and suppose a map $\epsilon \co \Cal{R} \rightarrow \Lambda$ is given. For an augmented simplicial commutative ring $A$ over $\Lambda$, define $\Lift_{\epsilon}(\Cal{R},A)$ to be the group of homotopy classes of lifts $\Cal{R} \rightarrow A$ lying over the given map $\epsilon$. 

For any discrete $\Lambda$-module $M$, we can take $A$ to be the square-zero extension $\Lambda \oplus M[i]$. (We remind the reader what this is: first, $\Lambda[i]$ is the free simplicial $\Lambda$-module on the simplicial set $S^i = \Delta^i/\partial \Delta^i$. Tensoring with $M$ gives a simplicial $\Lambda$-module $M[i]$, and then the simplicial $\Lambda$-algebra $\Lambda \oplus M[i]$ is obtained by forming the square-zero extension level-wise.) There is a bilinear pairing 
\begin{equation}\label{eq: homotopy to cotangent}
\pi_i (\Cal{R}) \times \Lift_{\epsilon}(\Cal{R}, \Lambda \oplus M[i]) \rightarrow M
\end{equation}
defined as follows: any $u \in \Lift_{\epsilon}(\Cal{R}, \Lambda \oplus M[i])$ induces 
\[
\pi_*(u) \co \pi_* (\Cal{R})  \rightarrow \pi_* (\Lambda \oplus M[i]) = \Lambda \oplus M[i]
\]
and \eqref{eq: homotopy to cotangent} takes $(x \in \pi_i (\Cal{R}), u)$ to $\pi_*(u)(x) \in M$. Note that $\Lift_{\epsilon}(\Cal{R}, \Lambda 	\oplus M[i])$ coincides with the Andr\'{e}-Quillen homology group $D_i^{\Z}(\Cal{R}; M)$, where $M$ is made an $\Cal{R}$-module via $\epsilon$. 

Let $\Lambda^{\vee}$ be the Pontrjagin dual to $\Lambda$. For $\Lambda = \Cal{O}$, this can be canonically identified with $\Frac(\Cal{O})/\Cal{O}$. The Galois representation $\rho_{\Cal{O}}$ induces a map $\rho \co \pi_0 (\Cal{R}_S) \rightarrow \Cal{O}$ since $\pi_0 (\Cal{R}_S)  = \mrm{R}_S$ is the usual Galois deformation ring. Taking $\Lambda = \Cal{O}$, \cite[Lemma 15.1]{GV} identifies $\Lift_{\rho}(\Cal{R}_S, \Lambda \oplus \Lambda[1])$ with $ H_f^2(\Z[1/S]; \Ad \rho_{\Cal{O}}) \cong \msf{V}^*$. Hence we get a map
\begin{equation}\label{eq: initial part} 
\pi_1 (\Cal{R}_S) \rightarrow H_f^2(\Z[1/S]; \Ad \rho_{\Cal{O}} \otimes \Lambda^{\vee})^*.
\end{equation}

Finally, composing \eqref{eq: initial part} with the identification of Poitou-Tate duality
\[
H_f^2(\Z[1/S]; \Ad \rho_{\Cal{O}} \otimes \Lambda^{\vee})^* \cong H_f^1(\Z[1/S]; \Ad \rho_{\Cal{O}}(1)) = \msf{V}^*
\]
gives the desired map $\pi_1(\Cal{R}_S) \rightarrow \msf{V}^*$; it is shown in \cite[Lemma 15.3]{GV} that this is an isomorphism. 

If we take $\Lambda = \Cal{O}/p^m\Cal{O}$ for some $m \geq 1$, then the representation $\rho_{\Lambda}$ obtained by reducing $\rho_{\Cal{O}}$ into $\wh{G}(\Lambda)$ induces a map $\Cal{R}_S \rightarrow \Lambda$. For the same reason as before, we obtain a map 
\[
\pi_1(\Cal{R}_S \dotimes \Lambda) \rightarrow H_f^2(\Z[1/S]; \Ad \rho_{\Lambda} \otimes \Lambda^{\vee})^* \cong \msf{V}^* \otimes_{\Cal{O}} \Lambda
\]
which is an isomorphism, by the case $\Lambda = \Cal{O}$ and our torsion-freeness assumptions.

\subsection{Formulation of derived local-global compatibility}

We now formulate a derived local-global compatibility statement which is analogous to \S \ref{ssec: GL local-global}. 
\begin{itemize}
\item The \emph{global automorphic object} is $\wt{\TT}_{\mf{m}}  \cong \wedge^* \msf{V}$, where this identification is the one from \cite[Proposition 8.6]{V}. 
\item The \emph{global spectral object} is $\pi_* (\Cal{R}_S) \cong \wedge^* \msf{V}^*$, where this identification is the one from \cite[Proposition 15.4]{GV}. Letting $\pi_* (\Cal{R}_S)^*$ be the $\Cal{O}$-dual of $\pi_* (\Cal{R}_S)$,  combining these identifications gives 
\begin{equation}\label{eq: global identifications}
\pi_* (\Cal{R}_S)^* \xrightarrow{\sim} \wt{\TT}_{\mf{m}}.
\end{equation}
\end{itemize}

\begin{remark}
The eventual local-global compatibility assertion in Theorem \ref{thm: compatibilty} does not depend on these choices. 
\end{remark}

The local actions that we want to compare are: 
\begin{itemize}
\item (Automorphic) The action of a local derived Hecke algebra $\Cal{H}_q(G(\Q_q); \Lambda)$ on $\wt{\TT}_{\mf{m}} \otimes_{\Cal{O}} \Lambda$ through the algebra homomorphism $\Cal{H}_q(G(\Q_q); \Lambda) \rightarrow \wt{\TT}_{\mf{m}} \otimes_{\Cal{O}} \Lambda$.
\item (Galois) The co-action of $\pi_*(\Cal{S}_q^{\Hk} \dotimes_{\Cal{O}} \Lambda)$ on $\pi_* (\Cal{R}_{S} \dotimes_{\Cal{O}} \Lambda) $, where the maps  coming from the co-action of $\Cal{S}_q^{\Hk}$ on $\Cal{R}_{S}$ and the fact that $\pi_*(\Cal{S}_q^{\Hk} \dotimes_{\Cal{O}} \Lambda)$ is free over $\mrm{S}_q^{\ur}\otimes_{\Cal{O}} \Lambda$ (so that the co-action descends to homotopy groups).
\end{itemize}

To state the comparison, it is convenient to dualize the co-action on the spectral side. By \eqref{eq: SHA = group homology} we have that
\begin{align*}
\Cal{S}_q^{\Hk} \dotimes_{\mrm{S}_q^{\ur}} \Cal{R}_{S} & \xrightarrow{\sim}  (
\Cal{O} \dotimes_{\Cal{O}[T_q]} \Cal{O} \dotimes_{\Cal{O}} \mrm{S}_q^{\ur}) \dotimes_{\mrm{S}_q^{\ur}} \Cal{R}_{S} \\
& \xrightarrow{\sim} (\Cal{O} \dotimes_{\Cal{O}[T_q]} \Cal{O}) \dotimes_{\Cal{O}} \Cal{R}_S.
\end{align*}


If $q \equiv 1 \in \Lambda$, then Example \ref{ex: degenerate edge map} applies above with $R =\Lambda$, $A = \Lambda\dotimes_{\Lambda[T_q]} \Lambda$, and $B = \Cal{R}_S \dotimes_{\Cal{O}} \Lambda$, giving 
\begin{align*}
\pi_*(\Cal{R}_S \dotimes_{\Cal{O}} \Lambda)  \xrightarrow{\mrm{co-act}} & \pi_* (\Cal{S}_q^{\Hk} \dotimes_{\mrm{S}_q^{\ur}} \Cal{R}_{S}  \dotimes_{\Cal{O}} \Lambda)  \\
 \cong & \pi_* (\Lambda \dotimes_{\Lambda[T_q]} \Lambda) \otimes_{\Lambda} \pi_*(\Cal{R}_S \dotimes_{\Cal{O}} \Lambda) .
\end{align*}
Dualizing over $\Lambda$, we then get an action 
\[
H^1(T_q; \Lambda) \otimes_{\Cal{O}} \pi_*(\Cal{R}_S)^* \rightarrow \pi_*(\Cal{R}_S)^*,
\]
where $\pi_*(\Cal{R}_S \dotimes_{\Cal{O}} \Lambda)^* = \Hom_{\Cal{O}}(\pi_*(\Cal{R}_S), \Lambda)$. To present this more symmetrically to the derived Hecke algebra, we use \eqref{cor: homotopy groups of SHA} to write 
\[
H^1(T_q; \Lambda) \otimes_{\Lambda} \pi_*(\Cal{R}_S \dotimes_{\Cal{O}} \Lambda)^* = \pi_* (\Cal{S}_q^{\Hk} \dotimes_{\Cal{O}} \Lambda)^*  \otimes_{\mrm{S}_q^{\ur} \otimes_{\Cal{O}} \Lambda}\pi_*(\Cal{R}_{S}\dotimes_{\Cal{O}} \Lambda)^*,
\]
where the homomorphism $\mrm{S}_q^{\ur} \rightarrow \pi_*(\Cal{R}_{S})^*$ corresponds to the character $\chi$. This is rather artificial of course: the usual (underived) local-global compatibility already intertwines the action on $\pi_0 (\Cal{S}_q^{\Hk} \dotimes_{\Cal{O}} \Lambda)^* = \mrm{S}_q^{\ur} \otimes_{\Cal{O}} \Lambda$ and $H_q(\Lambda)$ through the (underived) Satake isomorphism \eqref{eq: localized Satake}. Anyway, the upshot is that we dualize the co-action to an action 
\begin{equation}\label{eq: galois action}
\pi_* (\Cal{S}_q^{\Hk} \dotimes_{\Cal{O}} \Lambda)^* \otimes_{\mrm{S}_q^{\ur}\otimes_{\Cal{O}} \Lambda} \pi_*(\Cal{R}_{S} \dotimes_{\Cal{O}} \Lambda)^*  \xrightarrow{\mrm{act}} \pi_*(\Cal{R}_S \dotimes_{\Cal{O}} \Lambda).
\end{equation}

Abbreviate $\Cal{H}_q(\Lambda) := \Cal{H}(G(\Q_q), \Lambda) $. We will compare \eqref{eq: galois action} to the derived Hecke action 
\begin{equation}\label{eq: hecke action}
\Cal{H}_q(\Lambda)^* \otimes_{H_q(\Lambda)} (\wt{\TT}_{\mf{m}} \otimes_{\Cal{O}} \Lambda) \xrightarrow{\mrm{act}} (\wt{\TT}_{\mf{m}} \otimes_{\Cal{O}} \Lambda)
\end{equation}
where $H_q(\Lambda) \cong \Lambda[X_*(T)]_{\mf{m}}$ is as in \S \ref{sssec: hecke localization}.

\begin{thm}\label{thm: compatibilty}
Under the identifications \eqref{eq: global identifications} and \eqref{eq: local identifications}, the two actions \eqref{eq: hecke action} and \eqref{eq: galois action} coincide for all $q$ congruent to 1 modulo a sufficient large (depending on $\Lambda)$ power of $p$.\footnote{As we shall see in the proof, the restriction to sufficiently large $q$ comes from our use of Venkatesh's Reciprocity Law \cite[Theorem 8.5]{V}. We expect that \cite[Theorem 8.5]{V} in fact holds for all Taylor-Wiles primes; if this were the case, then our proof would apply for all Taylor-Wiles primes as well.} In other words, for all $q$ congruent to 1 modulo a sufficient large (depending on $\Lambda)$ power of $p$, the following diagram commutes (compare \S \ref{ssec: GL local-global}):
\[
\begin{tikzcd}[column sep = tiny, row sep = tiny]
&  [-25pt] &  \pi_* (\Cal{S}_q^{\Hk} \dotimes_{\Cal{O}} \Lambda)^* \ar[rrrr, "\sim"', "{\eqref{eq: local identifications}}"] & & & & \Cal{H}_q(\Lambda)  \\
& \eqref{eq: galois action} & \acts & & & & \acts & \eqref{eq: hecke action}\\
& & \pi_*(\Cal{R}_S \dotimes_{\Cal{O}} \Lambda )^* \ar[rrrr, "\sim"', "{\eqref{eq: global identifications}}"] & & & &  (\wt{\TT}_{\mf{m}} \otimes_{\Cal{O}} \Lambda)
\end{tikzcd}
\]
\end{thm}

\begin{remark}
In defining $\Cal{S}_q^{\ur}$, we made an auxiliary choice of an element $\ol{\rho}^{\wh{T}}(\Frob_q) \in \wh{T}(k)$. Since the derived Hecke algebra and its action do not depend such an auxiliary choice, Theorem \ref{thm: compatibilty} shows that $\Cal{S}_q^{\ur}$ and its action are similarly independent of this choice. 
\end{remark}

We now make some initial reductions for the proof of Theorem \ref{thm: compatibilty}.

\subsubsection{Reduction to degrees $1$ and $2$}

Since \eqref{eq: localized DHA} and Theorem \eqref{eq: local identifications} imply that $\Cal{H}_q$ and $\pi_* (\Cal{S}_q^{\Hk} \dotimes_{\Cal{O}} \Lambda)^* $ are generated in degrees $1$ and $2$ over their degree $0$ subrings, it suffices to check that Theorem \ref{thm: compatibilty} is correct in degrees 1 and 2. In other words, we need to check that 
\begin{itemize}
\item The map 
\begin{equation}\label{eq: hecke deg 1}
\Cal{H}_q(\Lambda)^1 \otimes_{H_q(\Lambda)}  (\wt{\TT}_{\mf{m}}^0 \otimes_{\Cal{O}} \Lambda)\rightarrow  (\wt{\TT}_{\mf{m}}^1 \otimes_{\Cal{O}} \Lambda)
\end{equation}
agrees under \eqref{eq: global identifications} with 
\begin{equation}\label{eq: galois deg 1}
\pi_1 (\Cal{S}_q^{\Hk} \dotimes_{\Cal{O}} \Lambda)^* \otimes_{\mrm{S}_q^{\ur}\otimes_{\Cal{O}} \Lambda} \pi_*(\Cal{R}_{S} \dotimes_{\Cal{O}} \Lambda)^*  \xrightarrow{\mrm{act}} \pi_{*+1}(\Cal{R}_S \dotimes_{\Cal{O}} \Lambda)^*.
\end{equation}

\item The map 
\begin{equation}\label{eq: hecke deg 2}
\Cal{H}_q(\Lambda)^2 \otimes_{H_q(\Lambda)}  (\wt{\TT}_{\mf{m}}^0 \otimes_{\Cal{O}} \Lambda) \rightarrow  (\wt{\TT}_{\mf{m}}^2 \otimes_{\Cal{O}} \Lambda)
\end{equation}
agrees under \eqref{eq: global identifications} with 
\begin{equation}
\pi_2 (\Cal{S}_q^{\Hk} \dotimes_{\Cal{O}} \Lambda)^* \otimes_{\mrm{S}_q^{\ur}\otimes_{\Cal{O}} \Lambda} \pi_*(\Cal{R}_{S} \dotimes_{\Cal{O}} \Lambda)^*  \xrightarrow{\mrm{act}} \pi_{*+2}(\Cal{R}_S \dotimes_{\Cal{O}} \Lambda)^*.
\end{equation}
\end{itemize}

\subsubsection{Reduction to the cohomology of the torus}

We already know that ``underived'', i.e. degree 0, part of $H_q(\Lambda) \subset \Cal{H}_q(\Lambda)$ acts on $(\wt{\TT}_{\mf{m}} \otimes_{\Cal{O}} \Lambda)$ through the character $\chi_{\Lambda}$, and that $\pi_0(\Cal{S}_q^{\Hk} \dotimes_{\Cal{O}} \Lambda) =  \mrm{S}_q^{\ur} \otimes_{\Cal{O}} \Lambda$ also acts through $\chi_{\Lambda}$, and that the two actions are intertwined by \eqref{eq: localized Satake}.

Also, \eqref{eq: localized DHA} shows that the degree-$i$ part $\Cal{H}_q(\Lambda)^i$ is generated over $H_q(\Lambda)$ by $H^i(T_q; \Lambda)$. Similarly, \eqref{eq: dual homotopy groups} shows that $\pi_i(\Cal{S}_q^{\Hk} \dotimes_{\Cal{O}} \Lambda)^*$ is generated over $\mrm{S}_q^{\ur}$ by $H^i(T_q; \Lambda)$. 

Hence it suffices to show that 
\begin{itemize}
\item The map 
\[
H^1(T_q; \Lambda) \otimes_{\Lambda}  (\wt{\TT}_{\mf{m}}^0 \otimes_{\Cal{O}} \Lambda) \rightarrow  (\wt{\TT}_{\mf{m}}^1 \otimes_{\Cal{O}} \Lambda)
\]
agrees under \eqref{eq: global identifications} with 
\[
H^1(T_q; \Lambda) \otimes_{\Lambda}  \pi_*(\Cal{R}_{S} \dotimes_{\Cal{O}} \Lambda)^*  \xrightarrow{\mrm{act}} \pi_{*+1}(\Cal{R}_S \dotimes_{\Cal{O}} \Lambda)^*.
\]
\item The map 
\[
H^2(T_q; \Lambda) \otimes_{\Lambda}   (\wt{\TT}_{\mf{m}}^0 \otimes_{\Cal{O}} \Lambda)\rightarrow  (\wt{\TT}_{\mf{m}}^2 \otimes_{\Cal{O}} \Lambda)
\]
agrees with 
\[
H^2(T_q; \Lambda) \otimes_{\Lambda}  \pi_*(\Cal{R}_{S} \dotimes_{\Cal{O}} \Lambda)^*  \xrightarrow{\mrm{act}} \pi_{*+2}(\Cal{R}_S \dotimes_{\Cal{O}} \Lambda)^*.
\]
\end{itemize}

\subsubsection{Reduction to the action on the cyclic vector}
We claim that it suffices to check that the actions agree on the given cyclic vector in $H^{j_0}(Y(K); \Lambda)_{\mf{m}}$. Indeed, the action of the local derived Hecke algebras $\Cal{H}_q(\Lambda)$, as $q$ varies over Taylor-Wiles primes, generates all of $H^{j_0+*}(Y(K); \Lambda)_{\mf{m}}$ by \cite[Theorem 8.5]{V}. Hence the same holds for the action of $\pi_* (\Cal{S}_q^{\Hk} \dotimes_{\Cal{O}} \Lambda)^*$ once we verify that the two actions agree on the cyclic vector. Furthermore, Theorem \ref{thm: derived sat isom} and \eqref{eq: localized Satake} show that $\Cal{H}_q(\Lambda)$ actions commute with each other, and similarly for $(\Cal{S}_q^{\Hk} \dotimes_{\Cal{O}} \Lambda)^*$.

In conclusion, to prove Theorem \ref{thm: compatibilty} we ``only'' need to check that: 
\begin{align}\label{eq: dual deg 1}
& \scalebox{0.9}{$H^1(T_q; \Lambda) \rightarrow \Cal{H}_q(\Lambda)^1 \rightarrow  (\wt{\TT}_{\mf{m}}^1 \otimes_{\Cal{O}} \Lambda) \xrightarrow{\text{\cite[Proposition 8.6]{V}}} \msf{V} \otimes_{\Cal{O}} \Lambda$}  \quad \text{ is dual to } \\ \nonumber
&\scalebox{0.9}{$ \msf{V}^* \otimes_{\Cal{O}} \Lambda  \xrightarrow{\eqref{eq: identify homotopy groups}} \pi_1 (\Cal{R}_S \dotimes_{\Cal{O}} \Lambda)  \rightarrow  \pi_1( \Cal{S}_q^{\Hk} \dotimes_{\Cal{O}} \Lambda)  \otimes_{\mrm{S}_q^{\ur}} \pi_0(\Cal{R}_S \dotimes_{\Cal{O}} \Lambda) \xrightarrow{\text{Cor. \ref{cor: homotopy groups of SHA}}} H_1(T_q; \Lambda)$},
\end{align}
and that 
\begin{align}\label{eq: dual deg 2}
& \scalebox{0.9}{$H^2(T_q; \Lambda) \rightarrow \Cal{H}_q(\Lambda)^2 \rightarrow (\wt{\TT}_{\mf{m}}^2 \otimes_{\Cal{O}} \Lambda) \xrightarrow{\text{\cite[Proposition 8.6]{V}}} \wedge^2 \msf{V}\otimes_{\Cal{O}} \Lambda $} \quad \text{ is dual to } \\ \nonumber
& \scalebox{0.9}{$ \wedge^2 (\msf{V}^*\otimes_{\Cal{O}} \Lambda) \xrightarrow{\eqref{eq: identify homotopy groups}} \pi_2 (\Cal{R}_S \dotimes_{\Cal{O}} \Lambda) \rightarrow  \pi_2 (\Cal{S}_q^{\Hk} \dotimes_{\Cal{O}} \Lambda)\otimes_{\mrm{S}_q^{\ur}} \pi_0(\Cal{R}_S \dotimes_{\Cal{O}} \Lambda)   \xrightarrow{\text{Cor. \ref{cor: homotopy groups of SHA}}} H_2(T_q; \Lambda)$}.
\end{align}

The proofs of \eqref{eq: dual deg 1} and \eqref{eq: dual deg 2} occupy the rest of the paper.

\subsection{Checking compatibility in degree 1}\label{ssec: compatible deg 1}
We check \eqref{eq: dual deg 1}. This amounts to showing that a certain map $H^1(T_q; \Lambda) \rightarrow \msf{V} \otimes_{\Cal{O}} \Lambda$ to be dual to a certain map $\msf{V}^* \otimes_{\Cal{O}} \Lambda  \rightarrow H_1(T; \Lambda)$, and we will now explicate what these maps are. 

\subsubsection{The automorphic side} We explicate the map $H^1(T_q; \Lambda) \rightarrow \msf{V} \otimes_{\Cal{O}} \Lambda$ from \eqref{eq: dual deg 1}. Recall that in \S \ref{sssec: fibration sequence} we defined a fiber sequence
\[
\Fib_{q,\Lambda} \rightarrow \Cal{F}_{\Z_q} \rightarrow \Cal{F}_{\Q_q}
\]
According to \cite[Theorem 8.5]{V}, there exists $n_0$ depending on $\Lambda$ such that for all $q \equiv 1 \pmod{n_0}$, the map $H^1(T_q;\Lambda) \rightarrow \msf{V} \otimes_{\Cal{O}} \Lambda$ can be described by the following sequence of steps.
\begin{enumerate}
\item The isomorphism $H^1(T_q;\Lambda) = \Hom(T_q; \Lambda) \cong \mf{t}_1(\Fib_{q,\Lambda})$ from \S \ref{sssec: t_1(Fib)}; this came from class field theory (describing tame deformations of a homomorphism into $\wh{T}$). 
\item The isomorphism $\mf{t}_1(\Fib_{q,\Lambda})  \xrightarrow{\sim} H^1(\Q_q; \Ad \rho_{\Lambda}) / H^1(\Z_q; \Ad \rho_{\Lambda}) $ from \eqref{eq: t_1 Fib map}.
\item The pairing
\[
\underbrace{H_f^1(\Z[1/S]; \Ad^* \rho_{\Lambda}(1))}_{(\msf{V} \otimes_{\Cal{O}} \Lambda)^*}  \times \frac{H^1(\Q_q; \Ad \rho_{\Lambda}) }{ H^1(\Z_q; \Ad \rho_{\Lambda}) } \rightarrow \Lambda
\]
given by restricting $H_f^1(\Z[1/S]; \Ad^* \rho_{\Lambda}(1)) \rightarrow H^1(\Q_q; \Ad^* \rho_{\Lambda}(1))$ and then applying Tate local duality. 
\end{enumerate}
This is summarized in the diagram 
\[
\begin{tikzcd}
H^1(T_q; \Lambda) \ar[r, "\text{\S \ref{sssec: t_1(Fib)}}", "\sim"'] & \mf{t}_1(\Fib_{q,\Lambda};\Lambda)  \ar[r, "\eqref{eq: t_1 Fib map}"]  & 
H^1(\Q_q; \Ad \rho_{\Lambda}) / H^1(\Z_q; \Ad  \rho_{\Lambda})  \\
 & \ar[r, "\text{local duality}"']  & 
(H_f^1(\Z[1/S]; \Ad^*\rho_{\Lambda}(1)))^* = \msf{V} \otimes_{\Cal{O}} \Lambda. \\
\end{tikzcd}
\]

\subsubsection{The Galois side} We describe the map $\msf{V} ^* \otimes_{\Cal{O}} \Lambda  \rightarrow H_1(T_q; \Lambda)$ from \eqref{eq: dual deg 2}. It comes from the sequence of steps:
\begin{enumerate}
\item The identification $\msf{V}^* \xrightarrow{\sim} \pi_1( \Cal{R}_{S})$ obtained by inverting \S \ref{sssec: pi_1 isom}.
\item The co-action map 
\[
\pi_1( \Cal{R}_{S}  \dotimes_{\Cal{O}} \Lambda ) \xrightarrow{\pi_1(\mrm{co-act})} \pi_1(\Cal{S}_q^{\Hk} \dotimes_{\mrm{S}_q^{\ur}}  \Cal{R}_{S}  \dotimes_{\Cal{O}} \Lambda  ).
\]
\item The projection map 
\[
\pi_1(\Cal{S}_q^{\Hk} \dotimes_{\mrm{S}_q^{\ur}}  \Cal{R}_{S}   \dotimes_{\Cal{O}} \Lambda  )\xrightarrow{\mrm{project}} \pi_1(\Cal{S}_q^{\Hk}   \dotimes_{\Cal{O}} \Lambda   ) \otimes_{\mrm{S}_q^{\ur} \otimes_{\Cal{O}} \Lambda } \pi_0(\Cal{R}_S \dotimes_{\Cal{O}} \Lambda  ).
\]
\item The identification $\pi_1(\Cal{S}_q^{\Hk}) \otimes_{\mrm{S}_q^{\ur} \otimes_{\Cal{O}} \Lambda} \pi_0(\Cal{R}_S  \dotimes \Lambda ) =  H_1(T_q; \Lambda)$ coming from Corollary \ref{cor: homotopy groups of SHA} and the assumption $\pi_0(\Cal{R}_S)   = \Cal{O}$. 
\end{enumerate}
This is summarized in the diagram 
\[
\begin{tikzcd}[column sep = small]
\msf{V}^* \otimes_{\Cal{O}} \Lambda \ar[r, "{\sim}"] &  \pi_1( \Cal{R}_{S} \dotimes_{\Cal{O}} \Lambda) \ar[r, "\text{co-act}"] & 
\pi_1(\Cal{S}_q^{\Hk} \otimes_{\Cal{S}_q}  \Cal{R}_{S} \dotimes_{\Cal{O}} \Lambda) \\
\ar[r, "\mrm{project}"] &   \pi_1(\Cal{S}_q^{\Hk} \dotimes_{\Cal{O}} \Lambda) \otimes_{\mrm{S}_q^{\ur} \otimes_{\Cal{O}} \Lambda} \pi_0(\Cal{R}_S \dotimes_{\Cal{O}} \Lambda)   \ar[r, "\text{Cor. \ref{cor: homotopy groups of SHA}}", "\sim"'] &  H_1(T_q; \Lambda).
\end{tikzcd}
\]

\subsubsection{Transfer to Andr\'{e}-Quillen cohomology}\label{sssec: homotopy to cotangent}
As discussed in \S \ref{sssec: pi_1 isom}, for any simplicial commutative ring $\Cal{R}$ with an augmentation to $\Lambda$, and a discrete $\Lambda$-module $M$, we have a pairing 
\[
\pi_i (\Cal{R}) \times D_{\Z}^i(\Cal{R}; M) \rightarrow M
\]
which induces a map 
\[
\pi_i(\Cal{R}) \rightarrow D_{\Z}^i(\Cal{R}; M)^* := \Hom_{\Lambda}(D_{\Z}^i(\Cal{R}; M),\Lambda).
\]
This is functorial in $\Cal{R}$, so we get a commutative diagram: 
\begin{equation}\label{eq: big diagram}
\begin{tikzcd}
& \msf{V}^* \otimes_{\Cal{O}} \Lambda \ar[d, equals]   \\
\pi_1(\Cal{R}_S \dotimes_{\Cal{O}} \Lambda ) \ar[d, "\text{co-act}"]  \ar[r] & D^1_{\Z}(\Cal{R}_S  ; \Lambda )^{*}  \ar[d, "\text{co-act}"] \\
\pi_1( \Cal{R}_S   \otimes_{\mrm{S}_q^{\ur}} \Cal{S}_q^{\Hk} \dotimes_{\Cal{O}} \Lambda )   \ar[d, "\mrm{project}"]  \ar[r] & D^1_{\Z}( \Cal{R}_S   \dotimes_{\mrm{S}_q^{\ur}} \Cal{S}_q^{\Hk};\Lambda)^{*}   \ar[d, "\mrm{project}"]  \\
\pi_1(\Cal{S}_q^{\Hk} \dotimes_{\mrm{S}_q^{\ur}} \pi_0\Cal{R}_S \dotimes_{\Cal{O}} \Lambda ) \ar[r] \ar[d, "\sim"', "\eqref{eq: SHA = group homology}"] &   D^1_{\Z}(\Cal{S}_q^{\Hk} \dotimes_{\mrm{S}_q^{\ur}} \pi_0\Cal{R}_S ;\Lambda )^{*}  \ar[d, "\sim"', "\eqref{eq: SHA = group homology}"]  \ar[dr, dashed] \\
\pi_1(\Lambda \dotimes_{\Lambda[T_q]} \Lambda) \ar[r] \ar[d, "\text{Cor. \ref{cor: homotopy groups of SHA}}", "\sim"'] & D_{\Z}^1(\Lambda \dotimes_{\Lambda[T_q]} \Lambda; \Lambda)^* \ar[d, "\sim"'] \ar[r, equals] & \mf{t}_1(\Fib_{q,\Lambda})^* \ar[d, "\sim"', "\S \ref{sssec: t_1(Fib)}"] \\
H_1(T_q; \Lambda)  \ar[r, equals] & H_1(T_q; \Lambda) \ar[r, equals]  & H_1(T_q; \Lambda) 
\end{tikzcd}
\end{equation}
Here: 
\begin{itemize}
\item The reason for commutivity for the second square is that it is actually obtained from a ring homomorphism 
\[
  \Cal{S}_q^{\Hk}  \dotimes_{\mrm{S}_q^{\ur}} \Cal{R}_S  \xrightarrow{\mrm{project}} \Cal{S}_q^{\Hk} \dotimes_{\mrm{S}_q^{\ur}} \pi_0\Cal{R}_S . 
\]
\item We used \eqref{eq: w.e. for Fib_q} to see that $\Lambda \dotimes_{\Lambda[T_q]} \Lambda$ represents $\Fib_{q, \Lambda}$.
\item We need to justify why the bottom left square in \eqref{eq: big diagram} commutes. By Proposition \ref{prop: hurewicz isom for SHA} the map $\pi_1(\Lambda \dotimes_{\Lambda[T_q]} \Lambda) \rightarrow  D_{\Z}^1(\Lambda \dotimes_{\Lambda[T_q]} \Lambda; \Lambda)^* $ is an isomorphism, but we have produced separate identifications of each with $H^1(T_q; \Lambda)$, and it is not entirely obvious that they are compatible. This is checked in \S \ref{ssec: compatibility of identifications}.
\end{itemize}



Upshot: since the bottom row in \eqref{eq: big diagram} is an isomorphism, and the top row is an isomorphism by \cite[Lemma 15.3]{GV}, the map of interest in \eqref{eq: homotopy to cotangent} is the same as the vertical composition along the right column in \eqref{eq: big diagram}.

\subsubsection{Some maps of tangent complexes}\label{sssec: tangent diagram chasing} We will now describe the dashed map in \eqref{eq: big diagram} in terms of a more general framework. 

Let $X,Y,Z$ be functors on artinian SCRs augmented over $\Lambda$, whose value on $\Lambda$ is contractible. We then have the theory of the tangent complex $\mf{t}_*$ for such functors \cite[\S 4 and Proof of Lemma 15.1]{GV}. For an augmented simplicial commutative ring $\Cal{R} \rightarrow \Lambda$, the $\mf{t}_i$ of the functor $\mrm{SCR}_{/\Lambda}(\Cal{R}, -)$ that $\Cal{R}$ represents coincides with the Andr\'{e}-Quillen cohomology $D^i_{\Z}(\Cal{R}; \Lambda)$. So we will also use $\mf{t}_*(\Cal{R})$ to denote $D^i_{\Z}(\Cal{R}; \Lambda) = \mf{t}_*( \mrm{SCR}_{/\Lambda}(\Cal{R}, -))$. 

Suppose we are given maps $X \rightarrow Z$ and $Y \rightarrow Z$. Let $F$ be the homotopy fiber of $Y \rightarrow Z$, i.e. $F = \Spec \Lambda \times^h_Z Y$. Then we have a diagram with all squares homotopy cartesian: 
\[
\begin{tikzcd}
F \ar[r] \ar[d] & Y \times_Z X \ar[r] \ar[d] & Y \ar[d] \\
\Spec \Lambda \ar[r] &  X \ar[r] & Z
\end{tikzcd}
\]
Hence we get a map 
\begin{equation}\label{eq: tangent fiber map}
\mf{t}_*(F) \rightarrow \mf{t}_*(Y \times_Z X).
\end{equation}

To describe this a little more explicitly, recall that the formation of tangent complexes preserves homotopy pullbacks (cf. \S \ref{ssec: tangent complex of Fib}), i.e. 
\begin{equation}\label{eq: tangent hofib}
\mf{t}_*(Y \times_Z X) = \mrm{hofib}(\mf{t}_* Y \oplus \mf{t}_* X \rightarrow	 \mf{t}_* Z).
\end{equation}
With respect to \eqref{eq: tangent hofib}, the map $\mf{t}_*(F) \rightarrow  \mf{t}_* Y \oplus \mf{t}_* X$ induced by \eqref{eq: tangent fiber map} is $0$ in the second coordinate and the tautological map induced by $Y \rightarrow F$ in the first coordinate.



\begin{example}
If we apply this discussion with $Y = \Cal{F}_{\Z_q, \rho_{\Lambda}}^{\wh{T}, \square}$, $Z = \Cal{F}_{\Q_q, \rho_{\Lambda}}^{\wh{T}, \square}$, and $X = \pi_0 \Cal{F}_{\Z[1/S], \rho_{\Lambda}}^{\mrm{crys}}$, then we get a map 
\[
\mf{t}_*(\Fib_q) \rightarrow \mf{t}_*(Y \times_Z X ) \xrightarrow{\sim} \mf{t}_*( (Y \times_Z Y) \times_Y \pi_0 \Cal{F}_{\Z[1/S]}^{\mrm{crys}}).
\]
Dualizing this recovers the map 
\begin{equation}\label{eq: cotangent fiber map}
\mf{t}_1(\Cal{S}_q^{\Hk} \dotimes_{\mrm{S}_q^{\ur}} \pi_0\Cal{R}_S)^* \xrightarrow{\eqref{eq: tangent fiber map}} \mf{t}_1(\Fib_{q,\Lambda})^*,
\end{equation}
which is the dashed arrow in \eqref{eq: big diagram}.
\end{example}

\subsubsection{Where are we?} 
We summarize the discussion with the diagram below. The map $H^1(T_q;\Lambda) \rightarrow \msf{V} \otimes_{\Cal{O}} \Lambda$ obtained by tracing along the right vertical edge of the diagram is the ``automorphic side'' of \eqref{eq: dual deg 1}, while the map $\pi_1(\Cal{R}_S \dotimes_{\Cal{O}} \Lambda) \rightarrow H_1(T_q;\Lambda)$ obtained by tracing along the left is ``Galois side'' of \eqref{eq: dual deg 1}.
\[
\begin{tikzcd}[column sep = small]
 \mf{t}_1(\Cal{R}_S \dotimes_{\Cal{O}} \Lambda )^*  \ar[d, "\text{co-act}"] \ar[r, equals] &  \msf{V}^*\otimes_{\Cal{O}} \Lambda \ar[r, dotted, dash] & H_f^1(\Z[1/S]; \Ad^* \rho_{\Lambda}(1))^* = \msf{V} \otimes_{\Cal{O}} \Lambda \\
 \mf{t}_1( \Cal{R}_S   \dotimes_{\mrm{S}_q^{\ur}} \Cal{S}_q^{\Hk} \dotimes_{\Cal{O}} \Lambda)^*   \ar[d, "\mrm{project}"]   & &  H^1(\Q_q; \Ad \rho_{\Lambda}) / H^1(\Z_q; \Ad \rho_{\Lambda})  \ar[u, "\text{local duality}"'] \\
  \mf{t}_1(\Cal{S}_q^{\Hk } \dotimes_{\mrm{S}_q^{\ur}} \pi_0(\Cal{R}_S) \dotimes_{\Cal{O}} \Lambda)^*   \ar[r, "\eqref{eq: cotangent fiber map}"]  & \mf{t}_1(\Fib_{q,\Lambda})^* \ar[d, "\sim"', "\text{CFT}"] \ar[r, dotted, dash] & \mf{t}_1(\Fib_{q,\Lambda})  \ar[u, "\eqref{eq: t_1 Fib map}"'] \\
 & H^1(T_q;\Lambda) \ar[r, dotted, dash] & H_1(T_q;\Lambda) \ar[u, "\text{CFT}"', "\sim"] 
\end{tikzcd}
\]
The dotted arrows connect spaces that are dual.

\subsubsection{Final steps} So we have reduced the content of the theorem to showing that the natural map 
\[
\mf{t}_1(\Cal{R}_S \dotimes_{\Cal{O}} \Lambda)^* \cong \msf{V}^*\otimes_{\Cal{O}} \Lambda   \rightarrow \mf{t}_1^*(\Fib_{q,\Lambda}),
\]
which ultimately came a general property of the structural setup, is dual to a map $\mf{t}_1(\Fib_{q,\Lambda}) \rightarrow \mf{t}_1(\Cal{R}_S \dotimes_{\Cal{O}} \Lambda)$ given by computing both in terms of Galois cohomology and then writing down a pairing using Tate local duality. 

This second map seems to have a more ``ad hoc'' description, but in the proof of \cite[Lemma 15.3]{GV} another description of it is given. Specifically, it is explained on \cite[p. 125]{GV} that this map pulls back to the map $\beta$ in \cite[eqn. (11.14)]{GV}, which means that it is the specialization of the map $\mf{t}_1(F) \rightarrow \mf{t}_1(X \times_Z Y)$ from \eqref{eq: tangent fiber map} to $X = \Cal{F}_{\Z[1/Sq], \rho_{\Lambda}}^{\crys}$, $Z = \Cal{F}_{\Q_q,  \rho_{\Lambda}}^{\wh{T}, \square}$, and $Y = \Cal{F}_{\Z_q,  \rho_{\Lambda}}^{\wh{T},\square}$. 

This observation reduces us to showing that in the situation of \S \ref{sssec: tangent diagram chasing}, the map $\mf{t}_1(F) \rightarrow \mf{t}_1(Y \times_Z X)$ from \eqref{eq: tangent fiber map} is dual to the one coming from the co-action:
\[
\begin{tikzcd}
\mf{t}_1(Y \times_Z X)^* \ar[r, "{\text{co-act}}"] &  \mf{t}_1((Y \times_Z Y) \times_Y (Y \times_Z X))^*  \\
\ar[r, "\text{project}"] & \mf{t}_1(Y \times_Z Y \times_Z \pi_0 X )^* \ar[r, "{\eqref{eq: tangent fiber map}}"] &  \mf{t}_1 (F)^*
\end{tikzcd}
\]
Here the last map $ \mf{t}_1(Y \times_Z Y \times_Z \pi_0 X )^* \rightarrow \mf{t}_1 (F)^*$ is an instance of \eqref{eq: tangent fiber map} but with the role of $X$ in \eqref{eq: tangent fiber map} played by $Y \times_Z \pi_0 X$. 

In other words, we've reduced to the claim that the following diagram commutes. 
\[
\begin{tikzcd}
& \mf{t}_1((Y \times_Z Y)\times_Y (Y \times_Z X)   ) \ar[dl, "\mrm{co-act}^*"]  \\
\mf{t}_1(Y \times_Z X)   & &  \mf{t}_1(Y \times_Z Y \times_Z \pi_0 X) \ar[ul, "\mrm{project}^*"]  \\
& \mf{t}_1(F) \ar[ur, "\eqref{eq: tangent fiber map}"'] \ar[ul, "{\eqref{eq: tangent fiber map}}"]
\end{tikzcd}
\]
This is verified by a direct inspection, using the explicit description of tangent complex of a fibered product \eqref{eq: tangent hofib}, and that \eqref{eq: tangent fiber map} is given by the ``tautological map into the first factor of $Y$''.

\subsection{Checking compatibility in degree 2} We next need to check \eqref{eq: dual deg 2}. Fortunately for us, this is more degenerate than the degree 1 case. 

The cup product furnishes a map $\wedge^2 H^1(T_q; \Lambda) \rightarrow H^2(T_q; \Lambda)$, and let $H^2(T_q; \Lambda)_{\mrm{ind}}$ be the quotient. The quotient map splits canonically by identifying $H^2(T_q; \Lambda)_{\mrm{ind}}$ as the primitive subspace of $H^2(T_q' \Lambda)$ for the coproduct induced by the group structure on $T_q$, inducing a direct sum decomposition
\[
H^2(T_q; \Lambda) \cong \wedge^2 H^1(T_q; \Lambda) \oplus H^2(T_q; \Lambda)_{\mrm{ind}}.
\]
Similarly we have 
\[
H_2(T_q; \Lambda) \cong \wedge^2 H_1(T_q; \Lambda) \oplus H_2(T_q; \Lambda)_{\mrm{prim}}.
\]
The compatibility in degree 1, which we just checked in \S \ref{ssec: compatible deg 1}, reduces us to checking \eqref{eq: dual deg 2} for the primitive/indecomposable parts: 
\begin{align}\label{eq: primitive dual deg 2}
& \scalebox{0.85}{$H^2(T_q; \Lambda)_{\mrm{ind}} \rightarrow \Cal{H}_q(\Lambda)^2 \rightarrow (\wt{\TT}_{\mf{m}}^2 \otimes_{\Cal{O}} \Lambda) \xrightarrow{\text{\cite[Proposition 8.6]{V}}}  \wedge^2 \msf{V}\otimes_{\Cal{O}} \Lambda$} \quad \text{ is dual to } \\ \nonumber
& \scalebox{0.85}{$\wedge^2 (\msf{V}^*\otimes_{\Cal{O}} \Lambda) \xrightarrow{\eqref{eq: identify homotopy groups}} \pi_2 (\Cal{R}_S \dotimes_{\Cal{O}} \Lambda) \rightarrow  \pi_2 (\Cal{S}_q^{\Hk} \dotimes_{\Cal{O}} \Lambda)\otimes_{\mrm{S}_q^{\ur}} \pi_0(\Cal{R}_S \dotimes_{\Cal{O}} \Lambda)   \xrightarrow{\text{Cor \ref{cor: homotopy groups of SHA}}} H_2(T_q; \Lambda)_{\mrm{prim}}$}.
\end{align}

\subsubsection{The automorphic side} We unravel the map $H^2(T_q; \Lambda)_{\mrm{ind}} \rightarrow \wedge^2 \msf{V} \otimes_{\Cal{O}} \Lambda$ from \eqref{eq: primitive dual deg 2}. In fact we claim that this map is $0$. In other words, we will argue that $H^2(T_q; \Lambda)_{\mrm{prim}}$ acts by $0$ on $H^*(Y(K); \Cal{O})_{\mf{m}}$. 

\begin{remark}
Note that \cite{V} actually ignores the part of the local derived Hecke algebra in degree $\geq 2$, using only $H^1$ to act on $H^*(Y(K); \Cal{O})_{\mf{m}}$. Our computation shows that in fact there is nothing to be gained at looking at the rest of the local derived Hecke algebras: all the non-trivial action comes from $H^1$. 
\end{remark}

Letting $\Lambda = \Cal{O}/p^m$, we have 
\[
H^2(T_q; \Lambda)_{\mrm{ind}} = \beta (H^1(T_q; \Lambda))
\]
where $\beta$ is the Bockstein operator associated to the short exact sequence
\begin{equation}\label{eq: bockstein sequence}
\Cal{O}/p^m \Cal{O} \rightarrow \Cal{O}/p^{2m} \Cal{O}  \rightarrow \Cal{O}/p^m \Cal{O}.
\end{equation}

Therefore our claim amounts to showing that the action of $\beta(a)$ on $H^*(Y(K); \Lambda)_{\mf{m}}$ is trivial for all $a \in H^1(T_q; \Lambda)$. Denote by $Y_0(q)$ the locally symmetric space obtained by adding $\Gamma_0(q)$-level structure to $Y(K)$, and let $\pi \co Y_0(q) \rightarrow Y(K)$ be the projection map. 

As defined above, $a$ and $\beta(a)$ are classes in $H^*(T_q; \Lambda)$. We will also use the notation $a$ and $\beta(a)$ to refer to their image in $\Cal{H}_q(\Lambda)$. We will use the notation $a'$ and $\beta(a')$ for their realization in $ H^1(Y_0(q); \Lambda)$ by pulling back via the map $Y_0(q) \rightarrow B(T_q)$ classifying the Shimura cover $Y_1^*(q) \rightarrow Y_0(q)$ (that is, the subcover of $Y_1(q) \rightarrow Y_0(q)$ with Galois group $T_q$).

The Iwahori Hecke algebra at $q$ with coefficients in $\Lambda$ acts on $H^*(Y_0(q);\Lambda)$. 
Recall that as part of the datum of a Taylor-Wiles prime we have an element $\Frob_q^{\wh{T}} \in \wh{T}(\Lambda)$. By \cite[Lemma 6.6 and the following discussion]{V}, we can view the element $\rho_{\Lambda}^{\wh{T}}(\Frob_q)$ as a character of the monoid algebra $\Lambda[X_*(T)^+]$ (which acts on $H^*(Y_0(q);\Lambda$ by what are usually called ``$U_q$ operators''). Hence the element $\Frob_q^{\wh{T}}$ cuts out a particular eigenspace of $H^*(Y_0(q);\Lambda)$. 

Recall that we have two different projection maps $\pi_1, \pi_2 \co Y_0(q) \rightrightarrows Y(K)$. By \cite[eqn. (144); cf. \S 8.16 and Lemma 8.17]{V}, the action of $\beta(a) \in \Cal{H}_q(\Lambda)^1$ on $H^*(Y(K); \Lambda)$ is given by:
\begin{quote}
Pullback (via $\pi_1$) to $Y_0(q)$, project to $\Frob_q^{\wh{T}}$-eigenspace, cup with $\beta(a')$, and pushdown (via $\pi_2$) to $Y(K)$. 
\end{quote}
In equations, $\beta(a) \in \Cal{H}_q(\Lambda)^1$ sends $y \in H^*(Y(K); \Lambda)_{\mf{m}}$ to 
\[
\pi_{2*} (\beta(a') \smile \Theta \star \pi_1^* (y))
\]
where $\Theta$ is the idempotent projector onto the $\Frob_q^{\wh{T}}$ eigenspace (the notation is chosen to match the $\Theta$ in \cite[Lemma 8.17]{V}). Since the Bockstein $\beta$ is a derivation with respect to the cup product, and commutes with finite pullbacks and pushforwards, we have
\begin{align}\label{eq: bocksteins}
\pi_{2*}( \beta(a') \smile x )& = \pi_{2*}(\beta(a' \smile \Theta \star \pi_1^*(y)) - a' \smile  \beta( \Theta \star \pi_1^* y)) \nonumber \\ 
&  = \pi_{2*}(\beta(a' \smile \Theta \star  \pi_1^*y)) -\pi_{2*}( a' \smile  \Theta \star \pi_1^* \beta( y)) 
\end{align}
Now, the commutative diagram 
\[
\begin{tikzcd}
\Cal{O} \ar[r, "p^m"] \ar[d] & \Cal{O} \ar[r] \ar[d] & \Cal{O}/p^m \Cal{O} \ar[d, equals] \\ 
\Cal{O}/p^m \Cal{O} \ar[r] &  \Cal{O}/p^{2m} \Cal{O}  \ar[r] &  \Cal{O}/p^m \Cal{O}
\end{tikzcd}
\]
shows that for any space $Y$, the Bockstein $\beta$ factors through 
\[
H^*(Y; \Cal{O}/p^m\Cal{O}) \rightarrow H^{*+1}(Y; \Cal{O})[p^m] \xrightarrow{\mrm{reduce}} H^{*+1}(Y; \Cal{O}/p^m\Cal{O}).
\]
Hence the first term $\pi_{2*}(\beta(a' \smile \Theta \star  \pi_1^*y))$ in \eqref{eq: bocksteins} is the reduction of a class in $H^*(Y(K); \Cal{O})[p^m]$, but this must vanish by our torsion-freeness assumption in \S \ref{ssec: automorphic side}. Similarly, the second term $\pi_{2*}( a' \smile  \Theta \star \pi_1^* \beta( y)) $ in \eqref{eq: bocksteins} vanishes because $\beta( y)$ already vanishes. 


\subsubsection{The Galois side} We unravel the map $\wedge^2 \msf{V}^*  \otimes_{\Cal{O}} \Lambda \rightarrow H_2(T_q ; \Lambda)_{\mrm{prim}}$ from \eqref{eq: primitive dual deg 2}. We must show that it is $0$. By definition, it comes from the sequence of steps:

\begin{enumerate}
\item The identification $\wedge^2 \msf{V}^*  \otimes_{\Cal{O}} \Lambda \xrightarrow{\sim} \pi_2( \Cal{R}_{S} \dotimes_{\Cal{O}} \Lambda)$ obtained by inverting \S \ref{sssec: pi_1 isom}.
\item The co-action map 
\[
\pi_2( \Cal{R}_{S} \dotimes_{\Cal{O}} \Lambda ) \xrightarrow{\mrm{co-act}} \pi_2(\Cal{S}_q^{\Hk} \dotimes_{\mrm{S}_q^{\ur}}  \Cal{R}_{S} \dotimes_{\Cal{O}} \Lambda ).
\]
\item The projection map 
\[
\pi_2(\Cal{S}_q^{\Hk} \dotimes_{\mrm{S}_q^{\ur}}  \Cal{R}_{S} \dotimes_{\Cal{O}} \Lambda  )\xrightarrow{\mrm{project}} \pi_2(\Cal{S}_q^{\Hk}\dotimes_{\Cal{O}} \Lambda  ) \otimes_{\mrm{S}_q^{\ur} \otimes_{\Cal{O}} \Lambda } \pi_0(\Cal{R}_S \dotimes_{\Cal{O}} \Lambda ).
\]
\item The identification $\pi_2(\Cal{S}_q^{\Hk}) \otimes_{\mrm{S}_q^{\ur}} \pi_0(\Cal{R}_S)   \otimes \Lambda  =  H_2(T_q; \Lambda)$ coming from \eqref{cor: homotopy groups of SHA} and the assumption $\pi_0(\Cal{R}_S)   = \Cal{O}$. 
\item The projection $H_2(T_q; \Lambda) \rightarrow H_2(T_q;\Lambda)_{\mrm{prim}}$. 
\end{enumerate}
This is summarized in the diagram 
\[
\begin{tikzcd}[column sep = tiny]
\wedge^2 \msf{V}^* \otimes_{\Cal{O}} \Lambda \ar[r, "{\sim}"] &  \pi_2( \Cal{R}_{S} \dotimes_{\Cal{O}} \Lambda) \ar[r, "\text{co-act}"] & 
\pi_2(\Cal{S}_q^{\Hk} \dotimes_{\Cal{S}_q}  \Cal{R}_{S} \dotimes_{\Cal{O}} \Lambda) \\
\ar[r, "\mrm{project}"] &   \pi_2(\Cal{S}_q^{\Hk} \dotimes_{\Cal{O}} \Lambda) \otimes_{\mrm{S}_q^{\ur}} \pi_0(\Cal{R}_S)  \ar[r, "\text{Cor.\ref{cor: homotopy groups of SHA}}", "\sim"' ] &  H_2(T_q; \Lambda) \ar[r] & H_2(T_q; \Lambda)_{\mrm{prim}}.
\end{tikzcd}
\]

\subsubsection{Transfer to Andre-Quillen homology}

By the same reasoning as for \eqref{eq: big diagram}, we have a commutative diagram
\begin{equation}\label{eq: deg 2 big diagram}
\begin{tikzcd}
\wedge^2 \msf{V}^*  \otimes_{\Cal{O}} \Lambda \ar[d, "\sim", "\S \ref{sssec: pi_1 isom}"'] &  \\
\pi_2(\Cal{R}_S \dotimes_{\Cal{O}} \Lambda ) \ar[d, "\text{co-act}"]  \ar[r] & D^2_{\Z}(\Cal{R}_S  ; \Lambda )^{*}  \ar[d, "\text{co-act}"] \\
\pi_2( \Cal{R}_S   \otimes_{\mrm{S}_q^{\ur}} \Cal{S}_q^{\Hk} \dotimes_{\Cal{O}} \Lambda )   \ar[d, "\mrm{project}"]  \ar[r] & D^2_{\Z}( \Cal{R}_S   \dotimes_{\mrm{S}_q^{\ur}} \Cal{S}_q^{\Hk};\Lambda)^{*}   \ar[d, "\mrm{project}"]  \\
\pi_2(\Cal{S}_q^{\Hk} \dotimes_{\mrm{S}_q^{\ur}} \pi_0\Cal{R}_S \dotimes_{\Cal{O}} \Lambda ) \ar[r] \ar[d, "\sim"', "\eqref{eq: SHA = group homology}"] &   D^2_{\Z}(\Cal{S}_q^{\Hk} \dotimes_{\mrm{S}_q^{\ur}} \pi_0\Cal{R}_S ;\Lambda )^{*}  \ar[d, "\sim"', "\eqref{eq: SHA = group homology}"]    \\
\pi_2(\Lambda \dotimes_{\Lambda[T_q]} \Lambda) \ar[r] \ar[d, "\text{Cor. \ref{cor: homotopy groups of SHA}}", "\sim"'] & D_{\Z}^2(\Lambda \dotimes_{\Lambda[T_q]} \Lambda; \Lambda)^*     \ar[d, "\S \ref{sssec: t_2(Fib)}", "\sim"'] \\
H_2(T_q; \Lambda)  &   H_2(T_q; \Lambda)_{\mrm{prim}}
\end{tikzcd}
\end{equation}

As described above, the map in \eqref{eq: primitive dual deg 2} is obtained by starting with $\wedge^2 \msf{V}^*  \otimes_{\Cal{O}} \Lambda $ and then tracing downwards along the left edge of the diagram, and then projection to $H_2(T_q; \Lambda)_{\mrm{prim}}$. By Proposition \ref{prop: hurewicz isom for SHA}, the map $\pi_2(\Lambda \dotimes_{\Lambda[T_q]} \Lambda) \rightarrow D_{\Z}^2(\Lambda \dotimes_{\Lambda[T_q]} \Lambda; \Lambda)^*$ is an isomorphism. Therefore, to show that \eqref{eq: primitive dual deg 2} is $0$ it suffices to show that tracing downwards along the right edge of the diagram also gives $0$. But by \cite[Lemma 15.1]{GV} we have
\[
D^2_{\Z}(\Cal{R}_S; \Lambda)\cong H_f^3(\Z[1/S]; \Ad \rho_{\Lambda})
\]
and the latter vanishes because its $\Lambda$-dual is a subspace of $H^0(\Z[1/S]; \Ad^* \rho_{\Lambda}(1))$ by global duality for Galois cohomology \cite[Theorem B.1]{GV}, which vanishes by our assumptions that $\ol{\rho}$ is irreducible, and $G$ is semisimple.

\appendix

\section{Some simplicial commutative algebra} 
\subsection{Free simplicial commutative algebras}


 Recall that the forgetful functor $U$ from simplicial commutative rings to simplicial sets admits a left adjoint $F$ which fits into a Quillen adjunction. Given a simplicial set $X_{\bu}$, we call $FX = \Z[X_{\bu}]$ the ``free simplicial commutative ring on $X_{\bu}$''. This can be described explicitly -- see \cite[\S 4.1]{Iy07}. The analogous facts hold for simplicial $R$-algebras. Given a discrete ring $R$, the ``free simplicial $R$-algebra on a generator degree $n$'' is obtained by taking the free $R$-algebra on a simplicial set corresponding to the $n$-sphere $S^n$, and more generally we can perform this construction iteratively to form a ``free simplicial $R$-algebra on a set of a generators''.

\begin{lemma}\label{lem: homotopy of free}
Let $R$ be a discrete ring and $R[x_1, y_2]$ the free simplicial commutative ring on a generator $x_1$ in degree $1$ and $y_2$ in degree $2$. Then 
\[
\pi_*(R) = \wedge_R^* \langle x_1 \rangle \otimes \Gamma_R^* \langle x_2 \rangle
\]
where $\Gamma_R^*$ denotes the divided power algebra. 
\end{lemma}

\begin{proof}
This follows from \cite[Corollary 7.30]{Qui68}.
\end{proof}

\begin{lemma}\label{lem: freeness of SHA}
Assume $q \equiv 1 \in \Lambda$. Then the algebra $\Cal{S}_q^{\Hk} \dotimes_{\Cal{O}} \Lambda$ is free over $\mrm{S}_q^{\ur} \otimes_{\Cal{O}} \Lambda$ on $r$ generators in degree $1$ and $r$ generators in degree $2$, where $r = \rank(G)$. 
\end{lemma}

\begin{proof}
By \eqref{eq: SHA = group homology} it suffices to show that $\Lambda \dotimes_{\Lambda[G]} \Lambda$ is free over $\Lambda$ on generators in degree $1$ and $2$, where $G = (\Z/p^n)^r$, and $\Lambda = \Z/p^m$ with $m \leq n$. 

By the compatibility of the claim with tensor products, we reduce to the case $r=1$, so $G = (\Z/p^n\Z)$. The group homology of cyclic groups is well-known, and in this case we have a $\Lambda$-algebra isomorphism. 
\begin{equation}\label{eq: homology of cyclic groups}
H_*(G; \Lambda) = \wedge_{\Lambda}^* \langle x_1 \rangle \otimes \Gamma_{\Lambda}^* \langle y_2 \rangle.
\end{equation}

The choice of generators $x_, x_2$ above induces a map from the free simplicial $\Lambda$-algebra on generators in $x_1'$ in degree 1 and $x_2'$ in degree 2: 
\[
\Lambda[x_1',y_2'] \rightarrow \Lambda \dotimes_{\Lambda[G]} \Lambda
\]
sending $x_1' \mapsto x_1$ and $x_2' \mapsto x$. This induces an isomorphism on homotopy groups by \eqref{eq: homology of cyclic groups} and Lemma \ref{lem: homotopy of free}, and is therefore a weak equivalence. 
\end{proof}

Now we contemplate the Hurewicz map from \S \ref{sssec: pi_1 isom} for $\Cal{S}_q^{\Hk}$. We take our augmentation to be the composition  
\[
\epsilon \co \Cal{S}_q^{\Hk} \rightarrow \pi_0(\Cal{S}_q^{\Hk} ) = \mrm{S}_q^{\ur} \xrightarrow{\chi} \Cal{O}.
\]
For a discrete $\Cal{O}$-module $M$, it gives a pairing
\begin{equation}\label{eq: hurewicz spectral hecke}
\pi_i(\Cal{S}_q^{\Hk} \dotimes_{\Cal{O}} \Lambda) \times  \Lift_{\epsilon}(\Cal{S}_q^{\Hk}, \Cal{O} \oplus  M[i]) \rightarrow  M.
\end{equation}
Note that $\Lift_{\epsilon}(\Cal{S}_q^{\Hk} , \Cal{O} \oplus  M[i])$ can be identified with $D_i^{\Z}(\Cal{S}_q^{\Hk}, M)$.

\begin{prop}\label{prop: hurewicz isom for SHA}
Assume $q \equiv 1 \in \Lambda$. Then the map 
\[
\pi_i(\Cal{S}_q^{\Hk} \dotimes_{\Cal{O}} \Lambda) \rightarrow \underbrace{D^i_{\Z}(\Cal{S}_q^{\Hk}; \Lambda)^*}_{\text{$\Lambda$-linear dual of $D^i_{\Z}(\Cal{S}_q^{\Hk}; \Lambda)$}},
\]
induced by \eqref{eq: hurewicz spectral hecke}, is an isomorphism for $i=1,2$. 
\end{prop}

\begin{proof}
By Lemma \ref{lem: freeness of SHA} and the fact that $\mrm{S}_q^{\ur}$ is free over $\Cal{O}$, it suffices to check that the analogous map
\begin{equation}\label{eq: free homotopy to AQ}
\pi_i(\Lambda[x_1, y_2]) \rightarrow D_{\Z}^i(\Lambda[x_1, y_2]; \Lambda)^* ,
\end{equation}
is an isomorphism for $i=1,2$. Note that 
\[
D^i_{\Z}(\Lambda[x_1, y_2]; \Lambda) \cong \Hom(\Lambda[x_1, y_2]; \Lambda \oplus \Lambda[i]).
\]
By freeness, a homomorphism $\Lambda[x_1, y_2] \rightarrow \Lambda \oplus \Lambda[i]$ is determined by where it sends $x_1, x_2$. This shows that \eqref{eq: free homotopy to AQ} is surjective in degrees $i=1,2$. Since all of these groups are isomorphic to $\Lambda$ by inspection, and this is finite, they are necessarily also isomorphisms. 
\end{proof}

\subsection{Compatibility of two identifications}\label{ssec: compatibility of identifications}

We will check that the diagram 
\[
\begin{tikzcd}
\pi_1(\Lambda \dotimes_{\Lambda[T_q]} \Lambda) \ar[r, "\S \ref{sssec: pi_1 isom}"] \ar[d, "\text{Cor. \ref{cor: homotopy groups of SHA}}", "\sim"'] & D_{\Z}^1(\Lambda \dotimes_{\Lambda[T_q]} \Lambda; \Lambda)^* \ar[d, "\sim"'] \ar[r, equals] & \mf{t}_1(\Fib_{q,\Lambda})^* \ar[d, "\sim"', "\S \ref{sssec: t_1(Fib)}"]  \\
H_1(T_q; \Lambda)  \ar[r, equals] & H_1(T_q; \Lambda) \ar[r, equals]  & H_1(T_q; \Lambda) 
\end{tikzcd}
\]
commutes. This is the bottom left subdiagram of \eqref{eq: big diagram}. (By Proposition \ref{prop: hurewicz isom for SHA} we know that the upper horizontal arrows are isomorphisms, but we are claiming that they are given by the \emph{identity} map under the vertical identifications.) 

The point is that we want to show that our identifications 
\[
\pi_1(\Lambda \dotimes_{\Lambda[T_q]} \Lambda)  \xrightarrow{\text{Cor. \ref{cor: homotopy groups of SHA}}} H_1(T_q; \Lambda) \quad \text{ and } \quad D_{\Z}^1(\Lambda \dotimes_{\Lambda[T_q]} \Lambda; \Lambda)^* \xrightarrow{\S \ref{sssec: t_1(Fib)}}  H^1(T_q; \Lambda)
\]
are intertwined by the map of \S \ref{sssec: pi_1 isom}.

Let us spell out the map $D_{\Z}^1(\Lambda \dotimes_{\Lambda[T_q]} \Lambda; \Lambda)^* \xrightarrow{\S \ref{sssec: t_1(Fib)}} H^1(T_q; \Lambda)$ in more detail. Let $\Lambda[\delta_n] = \Lambda \oplus \Lambda[n]$ be the square-zero extension in degree $n$, as in \S \ref{sssec: pi_1 isom}. Then we have (cf. \cite[proof of Lemma 3.11]{GV})
\begin{equation}\label{eq: square-zero loop}
\Lambda \stackrel{h}\times_{\Lambda[\delta_n]} \Lambda \approx \Lambda[\delta_{n-1}].
\end{equation}
What was used in \S \ref{sssec: t_1(Fib)} is that $D_{\Z}^1(\Lambda \dotimes_{\Lambda[T_q]} \Lambda; \Lambda)  \cong D_{\Z}^0(\Lambda[T_q]; \Lambda)$, which we now explicate:
\begin{align*}
D_{\Z}^0(\Lambda[T_q]; \Lambda)   & \cong 
\Lift_{\epsilon}(\Lambda[T_q], \Lambda[\delta_0]) \\
[\eqref{eq: square-zero loop} \implies] & \cong \Lift_{\epsilon}(\Lambda[T_q], \Lambda \stackrel{h}\times_{\Lambda[\delta_1]} \Lambda ) \\
[\text{universal property} \implies] & \cong \pi_0(\msf{pt} \stackrel{h}\times_{\msf{SCR}_{\epsilon}(\Lambda[T_q],\Lambda[\delta_1])} \msf{pt}) \\
&= \Lift_{\epsilon}(\Lambda \dotimes_{\Lambda[T_q]} \Lambda, \Lambda[\delta_1]) \\
 & \cong  D_{\Z}^1(\Lambda \dotimes_{\Lambda[T_q]} \Lambda; \Lambda)
\end{align*}
Writing $I \subset \Lambda[T_q]$ for the augmentation ideal over $\Lambda$, we have 
\[
\Lift_{\epsilon}(\Lambda[T_q], \Lambda[\delta_0]) \xrightarrow{\sim} \Hom(I/I^2, \Lambda)
\]
by restricting an augmented homomorphism $\Lambda[T_q] \rightarrow \Lambda[\delta_0]$ to $I$, where it factors through $I/I^2$. In turn, $\Hom(I/I^2, \Lambda)$ is identified $H^1(T_q; \Lambda)$ via the isomorphism $I/I^2 \xrightarrow{\sim} T_q \otimes_{\Z} \Lambda$ sending $[t]-[e] \mapsto t \otimes 1$. 

Next we recall how we are identifying $\pi_1(\Lambda \dotimes_{\Lambda[T_q]} \Lambda)^*  = H^1(T_q; \Lambda) \xrightarrow{\sim} I/I^2$. This comes from the homotopy fiber sequence of simplicial $\Lambda$-modules
\[
\begin{tikzcd}
I \dotimes_{\Lambda[T_q]} \Lambda  \ar[r] & \Lambda \ar[r] & \Lambda \dotimes_{\Lambda[T_q]} \Lambda 
\end{tikzcd}
\]
which induces
\[
\pi_1(\Lambda \dotimes_{\Lambda[T_q]} \Lambda) \xrightarrow{\sim} \pi_0(I \dotimes_{\Lambda[T_q]} \Lambda)  \xrightarrow{\sim} I/I^2 \otimes_{\Z} \Lambda \xrightarrow{\sim} T_q \otimes_{\Z} \Lambda.
\]

Finally, we will compare these identifications under the map $\pi_1(\Lambda 
\dotimes_{\Lambda[T_q]} \Lambda) \xrightarrow{\S \ref{sssec: pi_1 isom}} D_{\Z}^1(\Lambda 
\dotimes_{\Lambda[T_q]} \Lambda; \Lambda)^*$. An element of $D_{\Z}^1(\Lambda \dotimes_{\Lambda[T_q]} \Lambda; \Lambda)$ is the homotopy class of a $\Lambda$-augmented homomorphism $f' \co \Lambda \dotimes_{\Lambda[T_q]} \Lambda \rightarrow \Lambda[\delta_1]$. As discussed above, the computation of $D_{\Z}^1(\Lambda \dotimes_{\Lambda[T_q]} \Lambda; \Lambda)$ is based on the equivalence between the datum of $f'$ and the datum of a map $f \co \Lambda[T_q] \rightarrow \Lambda[\delta_0]$, which is equivalent to a map $I/I^2 = T_q \otimes_{\Z} \Lambda \rightarrow \Lambda$. We need to compute the effect of the map
\begin{equation}\label{eq: pi_1 map}
\pi_1(f') \co \pi_1(\Lambda \dotimes_{\Lambda[T_q]} \Lambda) \rightarrow \pi_1(\Lambda[\delta_1]) = \Lambda.
\end{equation}
For this we can forget the ring structure and compute at the level of simplicial $\Lambda$-modules. Then we have two exact triangles of simplicial $\Lambda$-modules:
\[
\begin{tikzcd}
I \dotimes_{\Lambda[T_q]} \Lambda  \ar[r] \ar[d] & \Lambda \ar[r] \ar[d] & \Lambda \dotimes_{\Lambda[T_q]} \Lambda \ar[d] \\
\Lambda[\delta_0] \ar[r] & \Lambda \ar[r] & \Lambda[\delta_1]
\end{tikzcd}
\]
and so \eqref{eq: pi_1 map} is identified with the map 
\[
\pi_0(f') \co \pi_0(I \dotimes_{\Lambda[T_q]} \Lambda )   \rightarrow  \wt{\pi}_0(\Lambda[\delta_0])   = \Lambda
\]
where $\wt{\pi}$ denotes reduced homology (i.e. removing the contribution from $\pi_0(\Lambda)$). This map can be read off from $f$: identifying $\pi_0(I \dotimes_{\Lambda[T_q]} \Lambda )  = I \otimes_{\Lambda[T_q] \Lambda} = I/I^2$, it is simply given by the restriction of $f$ to $I$ (which then factors through $I/I^2$). After a bit of unwrapping, one finds that this is exactly the desired compatibility. 





\bibliographystyle{amsalpha}
\bibliography{Bibliography}

   \end{document}